\documentclass[11pt]{article}

\PassOptionsToPackage{numbers, compress}{natbib}
\usepackage{amsmath,amssymb,amsthm,fullpage,mathrsfs,pgf,tikz,caption,subcaption,mathtools}
\usepackage[ruled,vlined,linesnumbered]{algorithm2e}
\usepackage{natbib}
\usepackage{amsmath,amssymb,amsthm,mathtools}
\usepackage[utf8]{inputenc} % allow utf-8 input
\usepackage[T1]{fontenc}    % use 8-bit T1 fonts
\usepackage{hyperref}       % hyperlinks
\usepackage{url}            % simple URL typesetting
\usepackage{booktabs}       % professional-quality tables
\usepackage{amsfonts}       % blackboard math symbols
\usepackage{nicefrac}       % compact symbols for 1/2, etc.
\usepackage{microtype}      % microtypography
\usepackage[margin=1in]{geometry}
\usepackage{bbm}
\usepackage{enumitem}
\usepackage{xcolor}
\usepackage{cleveref}

\usepackage{pdfcolmk}
\usepackage{babel}
\usepackage{enumitem}
\usepackage{amsbsy}
\usepackage{amstext}
\usepackage{amsthm}
\usepackage{amssymb}
\PassOptionsToPackage{normalem}{ulem}
\usepackage{ulem}
\numberwithin{equation}{section}
\numberwithin{figure}{section}
\numberwithin{table}{section}
\theoremstyle{plain}
\newtheorem{thm}{\protect\theoremname}[section]
\theoremstyle{plain}
\newtheorem{prop}[thm]{\protect\propositionname}
\theoremstyle{definition}
\newtheorem{conjecture}[thm]{\protect\conjecturename}
\theoremstyle{plain}
\newtheorem{lem}[thm]{\protect\lemmaname}
\theoremstyle{plain}
\newtheorem{rem}[thm]{\protect\remarkname}
\theoremstyle{definition}
\newtheorem{defn}[thm]{\protect\definitionname}

\usepackage{dsfont}
\usepackage{tikz}

\makeatother

\providecommand{\definitionname}{Definition}
\providecommand{\lemmaname}{Lemma}

\providecommand{\propositionname}{Proposition}
\providecommand{\remarkname}{Remark}
\providecommand{\theoremname}{Theorem}
\providecommand{\conjecturename}{Conjecture}

\DeclareMathOperator{\Cayley}{Cayley}
\DeclareMathOperator{\id}{id}
\DeclareMathOperator{\mchar}{char}
\DeclareMathOperator{\SL}{SL}

\DeclareMathOperator{\PGL}{PGL}
\DeclareMathOperator{\PSL}{PSL}

\global\long\def\C{\mathbb{C}}%
\global\long\def\Q{\mathbb{Q}}%
\global\long\def\N{\mathbb{N}}%
\global\long\def\Z{\mathbb{Z}}%
\global\long\def\O{\mathcal{O}}%
\global\long\def\F{\mathbb{F}}%
\global\long\def\A{\mathbb{A}}%
\global\long\def\one{\mathds{1}}%

\global\long\def\n#1{\left\lVert #1\right\rVert }%
\global\long\def\Mod#1{\ (\mathrm{mod}\ #1)}%

\begin{document}
\title{%Number Theoretic Ramanujan Graphs are not Good Vertex Expanders \\
%Combinatorics via Closed Orbits: \\ Vertex Expansion and Graph Quantum Ergodicity \\
%Number theoretic Ramanujan graphs are not necessarily unique neighbor expanders via the closed orbit paradigm
Combinatorics via Closed Orbits: \\ Number Theoretic Ramanujan Graphs are not Unique Neighbor Expanders 
}
\author{
Amitay Kamber
\footnote{
Centre for Mathematical Sciences, Wilberforce Road, Cambridge CB3 0WB, UK.
email: ak2356@dpmms.cam.ac.uk. Research supported by ERC under the European Union's Horizon 2020 research and innovation programme (grant agreement No. 803711)}
\and
Tali Kaufman
\footnote{Department of Computer Science, Bar-Ilan University, Ramat-Gan, 5290002, Israel, email: kaufmant@mit.edu, research supported by ERC.}
}
%\author{Amitay Kamber and Tali Kaufman}
\maketitle

\begin{abstract}
The question of finding expander graphs with strong vertex expansion properties such as \emph{unique neighbor expansion} and \emph{lossless expansion} is central to computer science. A barrier to constructing these is that strong notions of expansion could not be proven via the spectral expansion paradigm. 
%However, it is difficult to study them, as their expansion could not be proven via the spectral expansion paradigm. %, so there was no systematic way to understand this phenomenon. 

A very symmetric and structured family of optimal spectral expanders (i.e., Ramanujan graphs) was constructed using number theory by Lubotzky, Phillips and Sarnak, and was subsequently generalized by others. We call such graphs \emph{Number Theoretic Ramanujan Graphs}. 
These graphs are not only spectrally optimal, but also posses strong symmetries and rich structure. Thus, it has been widely conjectured that number theoretic Ramanujan graphs are lossless expanders, or at least unique neighbor expanders. 

In this work we disprove this conjecture, by showing that there are number theoretic Ramanujan graphs that are not even unique neighbor expanders. This is done by introducing a new combinatorial paradigm that we term \emph{the closed orbit method}. 

The closed orbit method allows one to construct finite combinatorial objects with extermal substructures. This is done by observing that there exist \emph{infinite} combinatorial structures with extermal substructures, coming from an action of a subgroup of the automorphism group of the structure.
The crux of our idea is a systematic way to construct a finite quotient of the infinite structure containing a \emph{simple} shadow of the infinite substructure, which maintains its extermal combinatorial property. %In particular, the shadow of a 'small' infinite substructure in the cover remains small in the quotient. 

Other applications of the method are to the edge expansion of number theoretic Ramanujan graphs and vertex expansion of Ramanujan complexes. Finally, in the field of graph quantum ergodicity we produce number theoretic Ramanujan graphs with an eigenfunction of small support that corresponds to the zero eigenvalue. This again contradicts common expectations. 

The closed orbit method is based on the well-established idea from dynamics and number theory of studying closed orbits of subgroups. The novelty of this work is in exploiting this idea to combinatorial questions, and we hope that it will have other applications in the future.
\end{abstract}

\section{Introduction}

\paragraph{On Ramanujan graphs, and Number Theoretic Ramanujan graphs.} Various combinatorial questions are studied using sparse graphs. Their solution is often based on a spectral analysis of the underlying graph, and in particular on the fact that the given graph is a good expander, meaning that the second eigenvalue of its adjacency matrix is very far from its first eigenvalue (\cite{hoory2006expander}). 
The strongest spectral expansion condition is the Ramanujan property, which says that the 
second largest non-trivial eigenvalue in absolute value of the adjacency operator of a $d$-regular graph is bounded by $2\sqrt{d-1}$. 

An explicit family of Ramanujan graphs, which we call \emph{LPS graphs}, was constructed in the celebrated work of Lubotzky, Phillips, and Sarnak (\cite{lubotzky1988ramanujan}). This construction is based on number theory and in particular on the theory of automorphic forms, using deep results of Deligne and others. There are various possible variations on the construction (e.g., \cite{margulis1988explicit,pizer1990ramanujan}), including the much earlier work of Ihara (\cite{ihara1966discrete}). We will focus on the work of Morgenstern (\cite{morgenstern1994existence}), who gave another such explicit family which we call \emph{Morgenstern graphs} (the essential difference between the two works is replacing the field $\Q$ by $\F_q(t)$, where $\F_q$ is the finite field with $q$ elements). We call the graphs resulting from the different variations \emph{number theoretic graphs}, to distinguish them from other constructions of Ramanujan graphs (e.g., the graphs constructed by \cite{marcus2013interlacing}). The number theoretic Ramanujan graphs have various other wonderful properties -- for example, they are Cayley Graphs and have a very large girth (i.e., the length of the shortest cycle is large).

\subsection{Vertex Expansion and the Spectral Method} 

There are some notoriously hard combinatorial questions about graphs where the spectral theory falls short of proving the desired answer. A notable example is the question of finding a family of explicit $d$-regular graphs which are \emph{lossless-expanders}. 
For $X$ a $d$-regular graph, and a subset $Y$ of the vertices of $X$, we define the \emph{expansion ratio} of $Y$ as $\frac{\left|N(Y)\right|}{\left|Y\right|}$, where $N(Y)$ is the set of neighboring vertices of the set $Y$, which may include vertices from $Y$ itself. This ratio is obviously bounded by $d$. For $d$ large but fixed as the size of the graph grows to infinity, we say that a family of graphs is a family of lossless expanders if there is a constant $\alpha>0$ such that for every set $Y\subset X$ of size $|Y| \le \alpha |X|$, its expansion ratio is $d-o(d)$. There are constructions of graphs satisfying weaker notions (\cite{capalbo2002randomness,alon2002explicit}), but even going beyond expansion ratio $d/2$ is a major open question (\cite{hoory2006expander}).

The best results using the spectral method are due to Kahale (\cite{kahale1995eigenvalues}). He shows that for a Ramanujan graph the expansion ratio of sets of size bounded by $\alpha |X|$ is at least $d/2 -\beta$, where $\beta \to 0$ as $\alpha \to 0$ (see Theorem~\ref{thm:Kahale vertex expansion} for an alternative proof). Kahale also constructs a family of graphs which are almost-Ramanujan, in the sense that their second largest eigenvalue in absolute value is bounded by $2\sqrt{d-1}+o(1)$, having a subset $Y$ of two vertices which has expansion ratio $d/2$. In particular, he shows that the best expansion ratio that it is possible to get solely by using spectral arguments cannot exceed $d/2$ for linear sized sets.

One of the reasons that passing the $d/2$ barrier is important is that graphs with vertex expansion greater than $d/2$ are also \emph{unique neighbor expanders}, for if a set $Y$ has an expansion ratio that is greater than $d/2$, then there exists a vertex that has a unique neighbor in $Y$. Unique neighbor expanders were constructed by Alon and Capalbo (\cite{alon2002explicit}), but the resulting graphs are not lossless expanders. A weaker desired property is \emph{odd neighbor expansion}, which says that there is a vertex that is connected to an odd number of elements in $Y$. We refer to \cite{hoory2006expander} for a discussion of vertex expansion and its applications from different points of view.

Kahale's example has a short cycle of length 4. It is also very far from being a Cayley graph. For graphs with large girth, Kahale actually proved that small sets have expansion ratio $d-o(d)$ (see also \cite{mckenzie2020high}). For the LPS graphs or Morgenstern graphs, which are the $d$-regular graph that have the best known girth, this implies that sets of size smaller than $|X|^{1/3-\epsilon}$ have expansion ratio close to $d$. 

The fact that Kahale's construction does not share many of the wonderful properties of the number theoretic construction led various researchers to speculate that Ramanujan number theoretic graphs, which have very large girth, should be lossless expanders, or at least graphs with vertex expansion strictly greater than $d/2$. 

We show that this common belief is not true. As a matter of fact, some Morgenstern Ramanujan graphs are not even \emph{odd neighbor expanders}, and therefore not \emph{unique neighbor expanders}. Here is one of our main theorems:

\begin{thm} [Number theoretic Ramanujan graphs that are not odd neighbor expanders]
\label{thm:vertex expansion intro}For every prime power $q$, there exists an infinite family of $(q+1)$-regular number theoretic Ramanujan graphs $X$, and a subset $Y\subset X$, $\left|Y\right|=O(\sqrt{|X|})$,
such that every $x\in N(Y)$ has precisely $2$ neighbors in $Y$. Therefore, $Y$ has no unique neighbors and $|N(Y)|=\frac{q+1}{2}|Y|$.

Explicitly, for every odd prime power $q$ and $m$
large enough, there exists a $\left(q+1\right)$-regular bipartite
Morgenstern Ramanujan graph $$X=\Cayley\left(\PGL_{2}\left(\F_{q^{2m}}\right),\{ \gamma_{1},...,\gamma_{q+1}\} \right),$$
with generators $\gamma_{1},...,\gamma_{q+1}$, such that the subgroup $\langle \gamma_{1}^{2},...,\gamma_{q+1}^{2}\rangle$ is isomorphic to $\PGL_{2}(\F_{q^{m}})$.
Moreover, the graph 
$$Y = \Cayley\left(\left\langle \gamma_{1}^{2},...,\gamma_{q+1}^{2}\right\rangle,\left\{ \gamma_{1}^{2},...,\gamma_{q+1}^{2}\right\} \right)$$
is also a $\left(q+1\right)$-regular bipartite Morgenstern Ramanujan graph.

The subset $Y \subset X$ is of size $|Y| =  O(\sqrt{|X|})$, and every $x\in N(Y)$ has precisely $2$ neighbors in $Y$. 
\end{thm}

The theorem is based on a new idea we introduce to Combinatorics called \emph{the closed orbit method}, which is based on working with the simply connected covering object, a topic we explain in the next subsection of this long introduction. 
In Subsection~\ref{subsec:vertex expansion intro} we describe how the general method applies to the vertex expansion question. We also discuss vertex expansion in Ramanujan complexes.
In Subsection~\ref{subsec:edge expansion intro} we discuss the problem of edge expansion in Ramanujan graphs and how our method addresses it.
In Subsection~\ref{subsec:quantum ergodicity intro} we describe the surprising application of our method to the field of quantum ergodicity of graphs, where we show in Theorem~\ref{thm:quantum ergodicity} the existence of a concentrated eigenfunction of the adjacency operator, again contradicting a natural belief that such eigenfunctions do not exist for number theoretic graphs.
In Subsection~\ref{subsec: closed orbit intro} we explain the closed orbit machinery in more detail, from a group theoretic point of view, and present our main abstract result, Theorem~\ref{thm:Abstract theorem}. 
Finally, in Subsection~\ref{subsec:vertex expansion conjecture} we discuss the inverse situation in which closed orbits do not exist. We pose the conjecture that the non-existence of closed orbits implies lossless expansion of number theoretic graphs.

\subsection{Combinatorics via the Covering object}

In our work, we study finite $d$-regular graphs by understanding new properties that they inherit from their infinite simply-connected cover, the $d$-regular tree $T_d$, together with the action of some group $G$ on it. We will usually call $T_d$ by $B_G$ below, since we think of it with the $G$-action. 
The infinite covering object is already evident in the definition of a Ramanujan graph -- it is a finite $d$-regular graph that inherits the spectrum of its covering object. Namely, the graph's non-trivial spectrum is contained in the spectrum of $T_d$.

The action of the group $G$ on $T_d$ is more subtle, but it underlies the number theoretic constructions of Ramanujan graphs. For example, the LPS construction (\cite{lubotzky1988ramanujan}) is based on the action of the $p$-adic group $G=\PGL_2(\Q_q)$ ($\Q_q$ is the $q$-adic field) on its \emph{Bruhat-Tits} tree $B_G = T_{q+1}$. Using number theory which is related to quaternion algebras, it is possible to construct an \emph{arithmetic lattice} $\Gamma$ in $G$ such that by taking the quotient of $T_{q+1}$ by $\Gamma$ we get a $(q+1)$-regular graph which inherits the spectrum of the infinite tree, namely, a graph with the Ramanujan property.

In our work, we focus on the action of a subgroup $H \le G$ on $B_G$. An orbit of the $H$-action on $B_G$ gives a substructure $Z \subset B_G$ with various desired properties. This substructure $Z$ is used to solve some combinatorial questions for the infinite cover. 

We then look at the projection of $Z$ into the finite quotient graph $X$. We want to understand the image $Y \subset X$ of the map, as it inherits the properties of $Z$. Usually, this map is very complicated, and in particular, its image $Y$ is the entire finite graph $X$. However, using the \emph{closed orbit method} that we introduce in this work, we show that there are special situations when this map is simple, and in particular, its image $Y$ may be small relative to $X$. 

As we explain in Subsection~\ref{subsec: closed orbit intro}, the special situations happen if the orbit $\Gamma\backslash H$ in the compact space $\Gamma \backslash G$ is closed, hence the name of the method\footnote{More generally, we are actually interested in a \emph{periodic orbit}, which is an H-orbit supporting a finite $H$-invariant measure. When $\Gamma\backslash G$ is compact, which is the case of interest to combinatorics, both notions are equivalent, and we think that the closed orbit method simply sounds better.}. The notion of a \emph{closed orbit} is basic in ergodic theory, and has many uses in homogeneous dynamics, number theory, and representation theory. The novelty of our work is exploiting this well-known notion to get a new understanding of finite combinatorial questions. 

All the above seems quite abstract, so let us now explain how we apply it to the problem of vertex expansion. This will require a more technical discussion.

%We use the closed orbit methodology to study several combinatorial problems. Most notable is the question of a graph with optimal vertex expansion (i.e., lossless expansion of graphs). It was a common belief among experts that number theoretic graphs are \emph{lossless expanders}. In Theorem~\ref{thm:vertex expansion intro} we disprove this speculation and show that there are number theoretic graphs that are not even \emph{unique neighbor expanders}, a weaker property that is implied by lossless expansion.

\subsection{Vertex Expansion and the Closed Orbit Method}\label{subsec:vertex expansion intro}

\paragraph{Using subgroups to find an infinite subgraph with bad vertex expansion.} 
Consider the field $\F_q((t))$ of Laurent series over the finite field $\F_q$ with $q$ elements ($q$ being a prime power). This field is analogous to the $q$-adic field $\Q_q$ when $\Q$ is replaced by the field $\F_q(t)$. As with the group $\PGL_2(\Q_q)$, the group $G=\PGL_2(\F_q((t)))$ acts naturally on a $(q+1)$-regular Bruhat-Tits tree $B_G$. Notice that in this case, there is a \emph{subfield} $\F_q((t^2))\subset\F_q((t))$. This subfield gives rise to a subgroup $H=\PGL_2(\F_q((t^2)))\le G$. Notice that the groups $G$ and $H$ are isomorphic, so $H$ acts on its own $(q+1)$-regular Bruhat-Tits tree $B_H$.

Next, consider the $H$-action on $B_G$, via the embedding of $H$ in $G$. An orbit of $H$ gives rise to an embedding of the vertices of the $(q+1)$-regular tree $B_H$ in $B_G$ (see Figure~\ref{fig:ramified extension tree} and
the discussion in Subsection~\ref{subsec:Bruhat-Tits Buildings}). The embedding can also be described as an embedding of the $(q+1,2)$-biregular
subdivision graph of $B_H$ in the $(q+1)$-regular
tree $B_G$. 

Notice that the image of the embedding is very thin, in the sense that a large ball in $B_G$ with $n$ vertices will contain $\Theta(\sqrt{n})$ vertices of $B_H$.
The following lemma says that the embedded set has bad expansion properties.

\begin{lem}[Lemma~\ref{lem:no unique neighbors}]\label{lem: bad vertex expansion in tree}
Let $Z\subset B_G$ be the embedding of the vertices of $B_H$ in $B_G$. Then each vertex $v\in N(Z)$ is a neighbor of precisely two vertices of $Z$.
\end{lem}

\begin{figure}
\begin{tikzpicture}

\tikzstyle{leftnode}=[draw= blue!80,text=blue!80, shape = circle, fill = blue!80,inner sep=2pt];
\tikzstyle{leftlabel} = [text = blue!80, scale = 1]
\tikzstyle{rightnode}=[draw= black!80,text=black!80, shape = circle, fill = black!80,inner sep=2pt];
\tikzstyle{rightlabel} = [text = black!80,scale = 1]

\node[leftnode,label={[leftlabel,left]$\boldsymbol{\left(\begin{array}{cc}1 & 0\\ & t^2\end{array}\right)}$}] (u0) at (0,1) {};
\node[leftnode,label={[leftlabel,left]$\boldsymbol{\left(\begin{array}{cc}1 & 0\\ & 1\end{array}\right)}$}] (u1) at (0,4) {};
\node[leftnode,label={[leftlabel,left]$\boldsymbol{\left(\begin{array}{cc}t^2 & 0\\ & 1\end{array}\right)}$}] (u2) at (0,8) {};
\node[leftnode,label={[leftlabel,above]$\boldsymbol{\left(\begin{array}{cc}t^2 & 1\\ & 1\end{array}\right)}$}] (u3) at (1,8) {};

\draw[draw = blue!80, very thick]
(u0) -- (u1)
(u1) -- (u2)
(u1) -- (u3)
;

\node[leftnode,label={[leftlabel,left]$\boldsymbol{\left(\begin{array}{cc}1 & 0\\ & t^2\end{array}\right)}$}] (v0) at (5,1) {};
\node[leftnode,label={[leftlabel,left]$\boldsymbol{\left(\begin{array}{cc}1 & 0\\ & 1\end{array}\right)}$}] (v1) at (5,4) {};
\node[leftnode,label={[leftlabel,left]$\boldsymbol{\left(\begin{array}{cc}t^2 & 0\\ & 1\end{array}\right)}$}] (v2) at (5,8) {};
\node[leftnode,label={[leftlabel,above]$\boldsymbol{\left(\begin{array}{cc}t^2 & 1\\ & 1\end{array}\right)}$}] (v3) at (8,8) {};
\node[rightnode,label={[rightlabel,left]$\boldsymbol{\left(\begin{array}{cc}1 & 0\\ & t\end{array}\right)}$}] (v4) at (5,2.5) {};
\node[rightnode,label={[rightlabel,right]$\boldsymbol{\left(\begin{array}{cc}t & 1\\ & t\end{array}\right)}$}] (v5) at (7,4) {};
\node[rightnode,label={[rightlabel,left]$\boldsymbol{\left(\begin{array}{cc}t & 0\\ & 1\end{array}\right)}$}] (v6) at (5,6) {};
\node[rightnode,label={[rightlabel,right = 17pt]$\boldsymbol{\left(\begin{array}{cc}t & 1\\ & 1\end{array}\right)}$}] (v7) at (7,6) {};
\node[rightnode,label={[rightlabel,above]$\boldsymbol{\left(\begin{array}{cc}t^2 & t\\ & 1\end{array}\right)}$}] (v8) at (6,8) {};
\node[rightnode,label={[rightlabel,above]$\boldsymbol{\left(\begin{array}{cc}t^2 & t+1\\ & 1\end{array}\right)}$}] (v9) at (10.5,8) {};

\draw[draw = blue!80, very thick]
(v0) -- (v4)
(v4) -- (v1)
(v1) -- (v6)
(v1) -- (v7)
(v6) -- (v2)
(v7) -- (v3)
;

\draw[very thick]
(v4) -- (v5)
(v6) -- (v8)
(v7) -- (v9)
;

\end{tikzpicture}

\caption{\label{fig:ramified extension tree}Part of the tree of $\protect\PGL_{2}\left(\protect\F_{2}\left(\left(t^{2}\right)\right)\right)$
(left) embedded in part of the tree of $\protect\PGL_{2}\left(\protect\F_{2}\left(\left(t\right)\right)\right)$
(right). See Subsection~\ref{subsec:Bruhat-Tits Buildings} for the
meaning of the vertex labels.}
\end{figure}

\paragraph{Using the closed orbit method to get a finite set with bad vertex expansion.}
Once we demonstrated the non-expanding set $Z\subset B_G$ in the infinite world, we may take a quotient of $B_G$ by a lattice, and look at the image $Y$ of $Z$ in the resulting finite graph. 

Since every neighbor of $Z$ is connected to $Z$ by at least two edges, the same is true for its image $Y\subset X$. Therefore, $Y$ has no unique neighbors, which is the property we are looking for. However, $Y$ may contain a lot of vertices, and even the whole of the graph $X$. The closed orbit method allows us to find an arithmetic quotient where $Y$ maintains its volume in the tree. Namely, we have the following: 

\begin{lem}[Special case of Theorem~\ref{thm:Abstract theorem}]\label{lem:small volume in quotient}
It is possible to choose a family of arithmetic lattices $\Gamma \le G$, such that the projection $Y$ of the set $Z$ into the finite graph $X= \Gamma \backslash B_G$ is of size $|Y| = O\left(\sqrt{|X|}\right)$. 
\end{lem}

Most of the non-explicit part of Theorem~\ref{thm:vertex expansion intro} follows from Lemma~\ref{lem:small volume in quotient} and the discussion above. The discussion implies that every vertex  $x \in N(Y)$ is connected by \emph{at least} two edges to $Y$. The fact that $x$ has precisely two neighbors in $Y$ follows from a symmetry trick we explain in Lemma~\ref{lem:Symmetry lemma}. This implies that there are number theoretic Ramanujan graphs that are not even odd neighbor expanders.  

In the following, we apply the closed orbit method to the Morgenstern Ramanujan graphs (\cite{morgenstern1994existence}). This will give \emph{explicit} number theoretic graphs with bad vertex expansion, and the explicit part of Theorem~\ref{thm:vertex expansion intro}.

\paragraph{ Morgenstern graphs: Explicit number theoretic graphs that are not lossless expanders.}

Let us describe how the above can be applied to the construction of Ramanujan graphs by Morgenstern (\cite{morgenstern1994existence}), for $q$ an odd prime power. Morgenstern constructs
a lattice $\Gamma\le\PGL_{2}\left(\F_{q}\left(\left(t\right)\right)\right)$
that acts simply transitively on the Bruhat-Tits tree $B_{G}$, with
generators $\gamma_{1},...,\gamma_{q+1}$. If $\Gamma_{n}$ is a normal subgroup of $\Gamma$, the graph $\Gamma_{n}\backslash B_{G}$ is then naturally isomorphic to the Cayley graph 
$X=\Cayley\left(\Gamma/\Gamma_{n},\left\{ \gamma_{1},...,\gamma_{q+1}\right\} \right)$.

When $\Gamma_{n}$ is chosen by some explicit congruence conditions
we get the Morgenstern graphs, which have plenty of nice properties,
described in \cite[Theorem 4.13]{morgenstern1994existence}, and are
very similar to the celebrated LPS graphs (\cite{lubotzky1988ramanujan,lubotzky1994discrete}).
In particular, $\Gamma/\Gamma_n$ is isomorphic to $\PGL_2(\F_{q^m})$, the graphs are Ramanujan graphs, and their girth is at least $4/3\log_{q}(|X|)$.

The general method applies as follows: It turns out that after a ``change of variables'', for $H=\PGL_{2}(\F_{q}((t^2)))$,
the subgroup $\Gamma\cap H$ is generated by $\delta_{1}=\gamma_{1}^{2},\dots,\delta_{q+1}=\gamma_{q+1}^{2}$
and is actually also a Morgenstern lattice of $H$, which acts simply transitively on the Bruhat-Tits tree $B_{H}$. After some computations, we end up with the explicit part of Theorem~\ref{thm:vertex expansion intro}.

\paragraph{Lossless expansion for Ramanujan complexes.}

A promising option for graphs with good vertex expansion are the underlying graphs of the Ramanujan complexes constructed in \cite{lubotzky2005ramanujan,li2004ramanujan} (see Subsection~\ref{subsec:Bruhat-Tits Buildings}
for some discussion of Ramanujan complexes).

Unlike $d$-regular graphs, higher dimensional Ramanujan complexes have a rigid local structure, which implies interesting combinatorial properties. For example, many recent works used the Garland method to show that they are \emph{high dimensional expanders} (see \cite{lubotzky2018high} and the references therein). Therefore, one may speculate that the rigid local structure will imply some form of lossless expansion. 

However, we show that the underlying graph on the vertices of the complex can have bad vertex expansion:
\begin{thm} [Skeleton graphs of Ramanujan complexes are not unique neighbor expanders]
\label{thm:Complexes}Let $n$ be prime, $q$ a prime power, $G=\PGL_{n}(\F_{q}((t)))$
and $B_{G}$ be the Bruhat-Tits building of $G$. Then there is an infinite family of Ramanujan quotients $\Gamma\backslash B_{G}$ such that its underlying
graph $X$ has a subset $Y$ of size $|Y|=O\left(|X|^{1/2}\right)$,
with no unique neighbors.
\end{thm}

The proof of Theorem~\ref{thm:Complexes} is based on applying the closed orbit method to $H=\PGL_n(\F_q((t^2)))\le G=\PGL_n(\F_q((t)))$. 
The basic observation is the following lemma:
\begin{lem}[Lemma~\ref{lem:no unique neighbors}]
The embedding $Z$ of the vertices of the building
of $B_H$ in the vertices of the building of $B_G$ have no unique neighbors.
\end{lem}

The lemma implies that the projection $Y \subset X$ of $Z \subset B_G$ will also have no unique neighbors. However, it does not have to be small. The closed orbit method allows us to find lattices such that the image $Y$ satisfies $|Y|=O(\sqrt{|X|})$.

\subsection{Bad Edge Expansion for Number Theoretic Graphs}\label{subsec:edge expansion intro}

The results about vertex expansion we described above have an analog for edge expansion. The \emph{edge expansion} of a set $S \subset X$ is the ratio $\frac{M(S,X-S)}{|S|}$, where $M(S,X-S)$ is the number of edges between $S$ and its complement in $X$. 
Another way of studying this ratio is by looking at the number $M(S,S)$ of internal edges in $S$, as $M(S,X-S)+M(S,S) = d|S|$. Finally, $\frac{M(S,S)}{|S|}$ is simply the average degree of the induced graph on $S$, which is the property we will actually study.

A graph is a good edge expander if the average induced degree for every small set $S$ is small. The best result about the connection between spectral gap and edge expansion is given in the following result of Kahale:
\begin{thm}[Kahale (\cite{kahale1995eigenvalues}), see also Theorem~\ref{thm:Kahale edge expansion}]
\label{thm:Kahale edge expansion intro}Let $X$ be a $\left(q^2+1\right)$-regular
Ramanujan graph, with $\left|X\right|\to\infty$. Then for every subset
$Y\subset X$ with $\left|Y\right|=o\left(\left|X\right|\right)$,
the average degree of the induced graph on $Y$ is bounded by $\sqrt{q}+1+o\left(1\right)$.
\end{thm}

As with vertex expansion, we can prove that there exist number theoretic graphs with as bad edge expansion as allowed by Kahale's result:
\begin{thm}[Kahale's spectral bound for edge expansion is tight]\label{thm:edge expansion}  For every prime power $q$, there exists
an infinite family of $\left(q^{2}+1\right)$-regular number theoretic
Ramanujan graphs $X$, and a $\left(q+1\right)$-regular induced subgraph
$Y\subset X$, $\left|Y\right|=O\left(\sqrt{\left|X\right|}\right)$.
\end{thm}

The proof of this theorem is based on applying the general construction
to $G=\PGL_{2}\left(\F_{q^{2}}\left(\left(t\right)\right)\right)$
and $H=\PGL_{2}\left(\F_{q}\left(\left(t\right)\right)\right)$. 
This gives as embedding $Z$ of $B_H$ in $B_G$. The basic property of this embedding is described in  Figure~\ref{fig:edge expansion} and the following Lemma:
\begin{lem}
Every vertex of $Z$ is connected to $q+1$ other vertices of $Z$.
\end{lem}
\begin{figure}
\begin{tikzpicture}

\tikzstyle{leftnode}=[draw= blue!80,text=blue!80, shape = circle, fill = blue!80,inner sep=2pt];
\tikzstyle{leftlabel} = [text = blue!80]
\tikzstyle{rightnode}=[draw= black!80,text=black!80, shape = circle, fill = black!80,inner sep=2pt];
\tikzstyle{rightlabel} = [text = black!80]

\node[leftnode,label={[leftlabel,left]$\boldsymbol{\left(\begin{array}{cc}1 & 0\\ & t^2\end{array}\right)}$}] (u0) at (0,0.5) {};
\node[leftnode,label={[leftlabel,left]$\boldsymbol{\left(\begin{array}{cc}1 & 0\\ & t\end{array}\right)}$}] (u1) at (0,2) {};
\node[leftnode,label={[leftlabel,left]$\boldsymbol{\left(\begin{array}{cc}1 & 0\\ & 1\end{array}\right)}$}] (u2) at (0,4) {};
\node[leftnode,label={[leftlabel,right]$\boldsymbol{\left(\begin{array}{cc}t & 1\\ & t\end{array}\right)}$}] (u3) at (1,4) {};
\node[leftnode,label={[leftlabel,left]$\boldsymbol{\left(\begin{array}{cc}t & 0\\ & 1\end{array}\right)}$}] (u4) at (0,6) {};
\node[leftnode,label={[leftlabel,above]$\boldsymbol{\left(\begin{array}{cc}t & 1\\ & 1\end{array}\right)}$}] (u5) at (1,6) {};

\draw[draw = blue!80, very thick]
(u0) -- (u1)
(u1) -- (u2)
(u1) -- (u3)
(u2) -- (u4)
(u2) -- (u5)
;

\node[leftnode,label={[leftlabel,left]$\boldsymbol{\left(\begin{array}{cc}1 & 0\\ & t^2\end{array}\right)}$}] (v0) at (6,0.5) {};
\node[leftnode,label={[leftlabel,left]$\boldsymbol{\left(\begin{array}{cc}1 & 0\\ & t\end{array}\right)}$}] (v1) at (6,2) {};
\node[leftnode,label={[leftlabel,left]$\boldsymbol{\left(\begin{array}{cc}1 & 0\\ & 1\end{array}\right)}$}] (v2) at (6,4) {};
\node[leftnode,label={[leftlabel,right]$\boldsymbol{\left(\begin{array}{cc}t & 1\\ & t\end{array}\right)}$}] (v3) at (7,4) {};
\node[leftnode,label={[leftlabel,left]$\boldsymbol{\left(\begin{array}{cc}t & 0\\ & 1\end{array}\right)}$}] (v4) at (6,6) {};
\node[leftnode,label={[leftlabel,above]$\boldsymbol{\left(\begin{array}{cc}t & 1\\ & 1\end{array}\right)}$}] (v5) at (7,6) {};
\node[rightnode,label={[rightlabel,right]$\boldsymbol{\left(\begin{array}{cc}t & \alpha\\ & t\end{array}\right)}$}] (v6) at (9,3.5) {};
\node[rightnode,label={[rightlabel,below = 8pt]$\boldsymbol{\left(\begin{array}{cc}t & \alpha + 1\\ & t\end{array}\right)}$}] (v7) at (11,3) {};
\node[rightnode,label={[rightlabel,above]$\boldsymbol{\left(\begin{array}{cc}t & \alpha\\ & 1\end{array}\right)}$}] (v8) at (9,6) {};
\node[rightnode,label={[rightlabel,below = 12pt]$\boldsymbol{\left(\begin{array}{cc}t & \alpha + 1\\ & 1\end{array}\right)}$}] (v9) at (10.5,6) {};

\draw[draw = blue!80, very thick]
(v0) -- (v1)
(v1) -- (v2)
(v1) -- (v3)
(v2) -- (v4)
(v2) -- (v5)
;
\draw[very thick]
(v1) -- (v6)
(v1) -- (v7)
(v2) -- (v8)
(v2) -- (v9)
;
\end{tikzpicture}

\caption{\label{fig:edge expansion}Part of the tree of $\protect\PGL_{2}\left(\protect\F_{2}\left(\left(t\right)\right)\right)$
(left) embedded in part of the tree of $\protect\PGL_{2}\left(\protect\F_{4}\left(\left(t\right)\right)\right)$
(right). We denote $\protect\F_{4}=\left\{ 0,1,\alpha,\alpha+1\right\} $.}
\end{figure}

When projected to a finite quotient, the image $Y$ of $Z$ still has an induced degree of at least $q+1$. The closed orbit method allows us to find an arithmetic lattice such that this projection is small.

\subsection{Concentrated eigenfunctions of number theoretic graphs}\label{subsec:quantum ergodicity intro}

The closed orbit machinery could be useful beyond the specific question of expansion.
Indeed, we use this idea to construct an eigenfunction with eigenvalue 0, which has small support, in a number theoretic graph.

There is a lot of recent work, initiated by Brooks and Lindenstrauss
(\cite{brooks2013non}), whose aim is to understand eigenfunctions of the adjacency operator on $\left(q+1\right)$-regular graphs. Similar to the setting of vertex expansion, eigenfunctions on a $\left(q+1\right)$-regular 
graph with girth at least $\beta\log_{q}n$ have support of size at least $\Theta\left(n^{\frac{\beta}{4}}\right)$ (see \cite[Subsection 1.1]{ganguly2018non}).
Therefore, the support of eigenfunctions on the graphs $X$ of Theorem~\ref{thm:vertex expansion intro} is of size at
least $n^{1/3-o\left(1\right)}$. More generally, Brooks and Lindenstrauss
(\cite{brooks2013non}) proved that for a $\left(q+1\right)$-regular
graph $X$ with girth $\beta\log_{q}n$, for every $\epsilon>0$ there
is $\delta>0$ such that if a set $Y$ supports $\epsilon$ of the
mass of an eigenfunction $f$ (where the eigenfunction is normalized
to $\n f_{2}=1$ and the mass is determined by $\left|f\right|^{2}$),
then $\left|Y\right|\ge\Omega_{\epsilon}\left(n^{\delta}\right)$.
This was improved by Ganguly and Srivastava (\cite{ganguly2018non})
to $\left|Y\right|=\Omega\left(\epsilon n^{\epsilon\beta/4}\right)$.

Recently, Alon, Ganguly and Srivastava (\cite{alon2021high}), extending the results of Ganguly and Srivastava (\cite{ganguly2018non}), constructed a family of $(q+1)$-regular graphs of high girth, with many eigenfunctions of small support, of eigenvalues that are dense in $\left(-2\sqrt{q},2\sqrt{q}\right)$.
Their graphs have second eigenvalue bounded by $\frac{3}{\sqrt{2}}\sqrt{q}\approx2.121\sqrt{q}$, which is close to being Ramanujan.

As for our contribution, let $X,Y$ be the graphs from Theorem~\ref{thm:vertex expansion intro}, with $X$ identified with $\PGL_2(\F_{q^{2m}})$ and $Y$ identified with $\PGL_2(\F_{q^{m}})$.
Let $f\colon X\to\C$ be
\begin{align*}
f\left(x\right) & =\begin{cases}
+1 & x\in\PSL_{2}\left(\F_{q^{m}}\right)\\
-1 & x\in\PGL_{2}\left(\F_{q^{m}}\right)-\PSL_{2}\left(\F_{q^{m}}\right)\\
0 & x\notin\PGL_{2}\left(\F_{q^{m}}\right)
\end{cases}.
\end{align*}
Recall that $Y$ is a bipartite graph. The function $f$ is simply the function giving the value $+1$ to one part of $Y$, the value $-1$ to the other part of $Y$, and the value $0$ for the vertices in $X-Y$.

\begin{thm}[Number theoretic Ramanujan graphs with concentrated eigenfunctions]\label{thm:quantum ergodicity}
The function $f\in L^{2}\left(X\right)$ is an eigenfunction of the
adjacency operator $A$ of $X$ with eigenvalue $0$.

Therefore, for every odd prime power $q$ there exists a $(q+1)$-regular number theoretic Ramanujan graph $X$ of girth greater than $4/3\log_q(|X|)$, with an eigenfunction of the adjacency operator of eigenvalue 0, which is supported on $O(\sqrt{|X|})$ vertices.
\end{thm}

\begin{proof}
Let $x\in X$. Notice that if $\gamma_{k}x\in Y$, then also $\gamma_{k}^{-1}x=\gamma_{k}^{-2}\gamma_{k}x\in Y$.
Moreover, $\gamma_{k}x$ and $\gamma_{k}^{-1}x$ are on different
parts of $Y$, as they differ by $\gamma_{k}^2$ which is a generator of $Y$ as a Cayley graph, and $Y$ is bipartite. Therefore, the total number of $+1$ contributions to
$\left(Af\right)\left(x\right)$ is equal to the total number
of $-1$ contributions to $\left(Af\right)\left(x\right)$. Therefore
$\left(Af\right)\left(x\right)=0$.
\end{proof}

Notice that after normalization our eigenfunction to $\n f_{2}=1$,
we have $\n f_{\infty}=\Omega\left(n^{-1/4}\right)$. By moving the
eigenfunction with the automorphisms of the Cayley graph, we actually
get $\Theta\left(\sqrt{n}\right)$ such functions. We are not familiar with any similar construction of an explicit non-trivial eigenfunction on number-theoretic graphs. However, our method is limited to the eigenvalue $0$.

We remark that our method is similar to the work of Mili\'{c}evi\'{c} about large values of eigenfunctions of arithmetic hyperbolic $3$-manifolds (\cite{milicevic2011large}), and also to the earlier work of Rudnick and Sarnak (\cite{rudnick1994behaviour}). Their methods show that an automorphic eigenfunction can have a large supremum-norm at closed orbits of smaller subgroups. The work \cite{milicevic2011large}, in particular, uses subgroups coming from field extensions. 
The main difference is that in our combinatorial setting we can explicitly construct the eigenfunction, and this eigenfunction is not automorphic in the sense that it is not an eigenfunction of the other Hecke operators that act on the space. However, the eigenvalue $0$ can be perhaps explained by the existence of an automorphic lift from the smaller group. It will be interesting to clarify this. 
Finally, it will be interesting to apply the methods of \cite{rudnick1994behaviour,milicevic2011large} to graphs, as they may prove the existence of more general eigenfunctions with a large supremum norm.

\subsection{The Closed Orbit Method}\label{subsec: closed orbit intro}

In the following, we explain the closed orbit method more accurately and state our abstract theorem about it. %We also describe our method in a different language, which may appeal to some readers.

Let $G$ be a locally compact group, $H\le G$ a closed subgroup and $\Gamma\le G$ a cocompact lattice. For $\Gamma x \in \Gamma \backslash G$, we may look at the $H$-orbit $\Gamma x H \subset \Gamma \backslash G$. The $H$-action defines a map 
\[
\tilde{F}_\Gamma : \Gamma_{x,H} \backslash H \to \Gamma \backslash G,
\]
where $\Gamma_{x,H} = x^{-1}\Gamma x \cap H$, given by sending $\Gamma_{x,H}h$ to $\Gamma xh$. Its image is $\Gamma x H \subset \Gamma \backslash G$.

An $H$-orbit, which is the image of this map, can be quite complicated topologically. However, when $\Gamma_{x,H}$ is a lattice in $H$, the map becomes much simpler, and in particular, it becomes a topological embedding, and its image is closed. We will focus of the case when $x= e \in G$ is the identity, and denote $\Gamma_H = \Gamma_{e,H}= \Gamma \cap H$. 
For our combinatorial purposes, we move from the group $G$ itself to a discrete space. We assume that $G$ and $H\le G$ are semisimple $p$-adic groups, and let $K\le G$ be a compact open subgroup. The space $G/K$ is a discrete space with a $G$-action, which is closely related to the Bruhat-Tits building $B_G$ of $G$. For simplicity, we will work with the space $G/K$ instead of the Bruhat-Tits building $B_G$. The left $H$-action on $G/K$ defines an embedding 
\[
  H/K_{H}\to G/K,
\]
where $K_H = H\cap K$. 

The reader may restrict herself to the case when $G = \PGL_n\left(\F_q((t))\right)$, $H = \PGL_n\left(\F_q((t^2))\right)$, $K= \PGL_n\left(\F_q[[t]]\right)$ and $K_H= \PGL_n\left(\F_q[[t^2]]\right)$. Then $G/K$ and $H/K_H$ may be identified with the vertices of the Bruhat-Tits buildings $B_G$ and $B_H$.

When we insert $\Gamma$ again into the picture, we get a map of discrete spaces
\[
F_{\Gamma}:\Gamma_H\backslash H/K_{H}\to\Gamma\backslash G/K.
\]
When $\Gamma_H$ is a lattice in $H$, this is a map between two finite combinatorial objects. 

For the applications, we want two properties: First, $\Gamma_H$ should indeed be a lattice in $H$. Second, we want $\Gamma_H\backslash H/K_{H}$ to be as small as possible relative to $\Gamma\backslash G/K$. 

To achieve the two properties we turn to number theory. Our general method will be:
\begin{enumerate}
    \item Construct an arithmetic lattice $\Gamma \le G$ such that $\Gamma_H $ is a lattice in $H$.
    \item Take congruence covers $\Gamma_n$ of $\Gamma$, such that the index $\left[\Gamma_H:\Gamma_{n}\cap \Gamma_H\right]$ will grow far slower than the index $\left[\Gamma:\Gamma_{n}\right]$. 
\end{enumerate}

We implement the above when $H$ and $G$ are related by field extension. The actual details are based on the theory of semisimple groups over adelic rings, and is done in Section~\ref{sec:The-Method}. Here is a non-precise version of our general abstract theorem. A precise
version is given in Theorem~\ref{thm:Good Pair Theorem}.

\begin{thm}[The Closed Orbit Method] 
\label{thm:Abstract theorem} Assume that $k_{0}$ is a non-Archimedean
local field, $l_{0}$ is a finite field extension of $k_{0}$ and
$\boldsymbol{G}$ is a semisimple algebraic group defined over $k_{0}$.
Let $H=\boldsymbol{G}\left(k_{0}\right)\le G=\boldsymbol{G}\left(l_{0}\right)$.
Then in many cases described in Section~\ref{sec:The-Method}, we
may choose a cocompact arithmetic lattice $\Gamma\le G$, such that
$\Gamma_H = \Gamma\cap H\le H$ is also a cocompact arithmetic lattice. Moreover,
we may choose a sequence $\left\{ \Gamma_{n}\right\} $ of principal
congruence subgroups of $\Gamma$ of growing index, such that

\[
\left[\Gamma_H:\Gamma_{n}\cap \Gamma_H\right]=O\left(\left[\Gamma:\Gamma_{n}\right]^{1/\left[l_{0}:k_{0}\right]}\right).
\]
Therefore, it holds that $\left|(\Gamma_{n}\cap \Gamma_H)\backslash H / K_H \right|=O\left(\left|\Gamma_{n}\backslash G / K\right|^{1/\left[l_{0}:k_{0}\right]}\right)$.
We conclude that we can construct a map between a small combinatorial object and a large combinatorial object, which locally looks like the embedding of $H/K_H$
in $G/K$.
\end{thm}

\subsection{Vertex Expansion without Closed Orbits}\label{subsec:vertex expansion conjecture}

Our explicit construction is based on the graphs of Morgenstern (\cite{morgenstern1994existence})
and not on the more famous LPS graphs of Lubotzky, Phillips and Sarnak
(\cite{lubotzky1988ramanujan}). As a matter of fact, our method completely fails for LPS
graphs, since they are based on lattices in $\PGL_{2}\left(\Q_{p}\right)$,
$p$ prime, and $\Q_{p}$ has no closed subfields. In particular,
$\PGL_{2}\left(\Q_{p}\right)$ has no closed subgroup, which behave similarly to the closed subgroup $\PGL_{2}\left(\F_{q}\left(\left(t^{2}\right)\right)\right)$
of $\PGL_{2}\left(\F_{q}\left(\left(t\right)\right)\right)$.

More generally, the LPS graphs are constructed from quaternion algebras over $\Q$, and $\Q$ has no subfields, which prevents us from applying the method we explain in Section~\ref{sec:The-Method}.

While our results may suggest that the LPS graphs may also have bad vertex expansion, we believe that the results actually point in the other direction. One can perhaps use the lack of similar subgroups to show that the LPS graphs have good vertex expansion, although implementing this idea seems hard.

Therefore, we end the introduction with the following conjecture:
\begin{conjecture}[LPS graphs are lossless expanders]
Let $q$ be fixed and large, and $X_n$ be the family of $(q+1)$-regular Ramanujan graphs constructed in \cite{lubotzky1988ramanujan}. Then for every $\epsilon>0$, there is $n$ large enough such that for every set $Y \subset X_n$ with $|Y| \le |X_n|^{1-\epsilon}$, we have 
\[
|N(Y)| \ge (q+1 - o(q)) |Y|.
\]
\end{conjecture}

\subsection*{Structure of this Article}

In Section~\ref{sec:Proof_of_Explicit_Theorem} we prove the explicit part of Theorem~\ref{thm:vertex expansion intro}, showing explicit number theoretic Ramanujan graphs with bad vertex expansion for odd prime powers. 
We assume the results of \cite{morgenstern1994existence}, the proof uses elementary number theory in function fields, and is independent of the rest of the paper.

In Section~\ref{sec:The-Method} we state and prove the precise version
of Theorem~\ref{thm:Abstract theorem}, which is our general theorem presenting the closed orbit method. The proof is based on the theory of semisimple groups over the adeles, and in particular on the strong approximation theorem.

In Section~\ref{sec:Division Algberas} we present the implications of the closed orbit method to vertex expansion and edge expansion.
We apply Theorem~\ref{thm:Abstract theorem} to division algebras,
discuss the Bruhat-Tits building, and prove Theorem~\ref{thm:edge expansion}
and Theorem~\ref{thm:Complexes}. We also prove the non-explicit part of Theorem~\ref{thm:vertex expansion intro}.

Finally, in Section~\ref{sec:Moore Bound},
we present simple proofs of Kahale's theorems, using the results
of~\cite{alon2002moore}. This section is independent of the other
sections.

\subsection*{Acknowledgement}
The authors wish to thank Joseph Bernstein, Alex Lubotzky, Sidhanth Mohanty, Peter Sarnak and Nikhil Srivastava for discussions surrounding this project. This work began when the first-named author was a PhD student in the Hebrew University of Jerusalem under Alex Lubotzky, and was supported by ERC grant 692854.

A shortened version of this worked appeared in the proceedings of STOC 2022.

\section{\label{sec:Proof_of_Explicit_Theorem}Explicit Number Theoretic Graphs with Bad Vertex Expansion}

The main goal of this section is to prove the explicit part (i.e., for \emph{odd} prime powers $q$) of Theorem~\ref{thm:vertex expansion intro}. The theorem is more than a special case of Theorem~\ref{thm:Abstract theorem}, since the lattices do not come from simply connected groups, as we assume (implicitly) in Theorem
~\ref{thm:Abstract theorem}. This allows us to construct very explicit Cayley graphs, but add another layer of complication, which we resolve using Morgenstern's results.

Let us first give a short explanation how the explicit construction fits into the general framework. Let $\Gamma=\langle \gamma_1 ,...,\gamma_{q+1} \rangle$ be a free group with $\gamma_1 ,...,\gamma_{q+1}$ as generators and their inverses. Let $\Gamma'$ be the subgroup of $\Gamma$ generated by $ \delta_1,...,\delta_{q+1} $, where $\delta_i = \gamma_i^2$. 

The Cayley graph $T_\Gamma = \Cayley(\Gamma,\{\gamma_1,...,\gamma_{q+1}\})$ is a $(q+1)$-regular tree. Similarly, $T_{\Gamma'} = \Cayley(\Gamma',\{\delta_1,...,\delta_{q+1}\})$ is also a $(q+1)$-regular tree. The embedding of $\Gamma'$ in $\Gamma$ gives an embedding of the vertices of $T_{\Gamma'}$ in $T_{\Gamma}$. Moreover, each edge in $T_{\Gamma'}$ corresponds to two edges in $T_{\Gamma}$, or alternatively, the embedding extends to a graph embedding of the $(q+1,2)$-biregular subdivision graph of $T_{\Gamma'}$ in $T_\Gamma$. We deduce that if $Z\subset T_\Gamma$ is the embedding of the vertices of $T_{\Gamma'}$ into $T_{\Gamma}$, then every vertex $v\in N(Z)$ is connected to two vertices of $Z$. This is a version of Lemma~\ref{lem: bad vertex expansion in tree} in our case.

Now, let $\Gamma_n$ be a finite index normal subgroup of $\Gamma$. Then we may look at the Cayley graph $X=\Cayley(\Gamma/\Gamma_n,\{\gamma_1,...,\gamma_{q+1}\}))$ (with the elements being identified with their image in $\Gamma/\Gamma_n)$. Alternatively, $X$ can be identified with the quotient of $T_\Gamma$ by $\Gamma_n$. There is a natural embedding $\Gamma'/(\Gamma'\cap \Gamma_n) \to \Gamma/\Gamma_n$. The image $Y$ of this embedding can be identified with the projection of $Z \subset T_\Gamma$ to $X$. We deduce that every neighbor of $Y$ is also connected to $Y$ by at least two edges. 

The problem is then to find a subgroup $\Gamma_n$ such that $|Y|$ will be much smaller than $|X|$, or alternatively $[\Gamma':\Gamma'\cap\Gamma_n]$ will be much smaller than $[\Gamma:\Gamma_n]$, which will be an explicit version of Lemma~\ref{lem:small volume in quotient}. 
During the proof, we will show that this holds for the Morgenstern graphs using explicit calculations.

It may be hard to identify the relation between the proof and the general theory, which involves $p$-adic groups. After setting some preliminaries we explain some of it in Remark~\ref{rem:explanation of explicit}, and later we explain another part of the connection in Subsection~\ref{subsec:explanation of explicit}. 

Throughout the proof, we freely use basic number theory in function
fields. See \cite{rosen2013number} for a good introduction to this
subject.

We start by recalling the construction of the Morgenstern Ramanujan graphs
(\cite{morgenstern1994existence}). Let $q$ be an odd prime power, with $\F_q$ the corresponding finite field. Consider the quaternion algebra
$A\left(\F_{q}\left(u\right)\right)$, which has a base $1,i,j,ij$
over $\F_{q}\left(u\right)$, with relations
\[
i^{2}=\epsilon,j^{2}=u-1,ij=-ji,
\]
where $\epsilon\in\F_{q}$ is a non-square. This algebra has a norm
\[
N\left(a+bi+cj+dij\right)=a^{2}-\epsilon b^{2}+\left(\epsilon d^{2}-c^{2}\right)\left(u-1\right).
\]

We let $A^{\times}\left(\F_{q}\left(u\right)\right)/Z^{\times}$ be
the quotient of the invertible elements of $A\left(\F_{q}\left(u\right)\right)$
by the equivalence condition $\alpha\sim\alpha'$ if and only if there
is $a\in\F_{q}\left(u\right)$ with $a\alpha=\alpha'$.

Denote by $A\left(\F_{q}\left[u\right]\right)$ the elements of $A$ with $a,b,c,d\in\F_{q}[u]$. There are $q+1$ elements $\left\{ \gamma_{1}^{\prime},...,\gamma_{q+1}^{\prime}\right\} \subset A\left(\F_{q}\left[u\right]\right)$,
satisfying
\[
\gamma_{k}^{\prime}=1+c_{k}j+d_{k}ij,
\]
with $c_{k},d_{k}\in\F_{q}$ and $N\left(\gamma_{k}^{\prime}\right)=u$.
Those elements correspond to the $q+1$ solutions of $\epsilon d^{2}-c^{2}=1$.

We let $S=\left\{ \gamma_{1},...,\gamma_{q+1}\right\} $ be the image
of those elements in $A^{\times}\left(\F_{q}\left(u\right)\right)/Z^{\times}$.
Finally, let $\Gamma$ be the group generated by $\gamma_{1},...,\gamma_{q+1}$.
\begin{thm}[{\cite[Corollary 4.7]{morgenstern1994existence}}]
$\Gamma$ is a free group on $\frac{q+1}{2}$ generators, with $\gamma_{1},...,\gamma_{q+1}$
as generators and their inverses. Moreover, there is a bijection between $\Gamma$ and the set
\begin{align*}
\Bigl\lbrace \alpha=a+bi+cj+dij\in A\left(\F_{q}\left[u\right]\right):&\exists  l\ge0,N\left(\alpha\right)=u^{l},\\
& u-1|\gcd\left(a-1,b\right),u\nmid\gcd\left(a,b,c,d\right)
\Bigr\rbrace.
\end{align*}
The bijection is given by choosing for every $\gamma\in \Gamma$ the unique element in its equivalence class in $A^\times(\F_q(u))$ from the set above.
\end{thm}

%The group structure on the right hand side is given by multiplication in $A\left(\F_{q}\left[u\right]\right)$, and then by dividing by the largest power of $u$ dividing $\gcd\left(a,b,c,d\right)$.

Now let $g\in\F_{q}\left[u\right]$ be an irreducible polynomial of
degree $2m$. Then the congruence subgroup $\Gamma\left(g\right)$
of $\Gamma$ is the set of elements of $\Gamma$ who have in their
equivalence class an element $a+bi+cj+dij\in A\left(\F_{q}\left[u\right]\right)$,
satisfying $g\nmid a,g|b,g|c,g|d$.

Recall that the Legendre symbol for $f,g\in\F_{q}\left(u\right)$,
$g\ne0$ irreducible is defined as
\begin{align*}
\left(\frac{f}{g}\right) & =\begin{cases}
0 & g|f\\
1 & f\text{ is a square\ensuremath{\ne0} }\operatorname{mod} g\\
-1 & \text{else}
\end{cases}\\
 & \equiv f^{\left(q^{\deg g}-1\right)/2}\Mod g.
\end{align*}

\begin{thm}[{\cite[Theorem 4.13]{morgenstern1994existence}}]
\label{thm:Morgenstern Theorem}The Cayley graph $X_{g}=\Cayley\left(\Gamma/\Gamma\left(g\right),\left\{ \gamma_{1},...,\gamma_{q+1}\right\} \right)$
is a $\left(q+1\right)$-regular Ramanujan graph.

There are two possibilities, depending on the Legendre symbol $\left(\frac{u}{g}\right)$:
\begin{enumerate}
\item If $\left(\frac{u}{g}\right)=-1$ then the group $\Gamma/\Gamma\left(g\right)$
is isomorphic to $\PGL_{2}\left(\F_{q^{2m}}\right)$, the graph $X_{g}$
is bipartite and its girth is at least $\frac{4}{3}\log_{q}\left(\left|X_{g}\right|\right)$.
\item If $\left(\frac{u}{g}\right)=1$ then the group $\Gamma/\Gamma\left(g\right)$
is isomorphic to $\PSL_{2}\left(\F_{q^{2m}}\right)$, the graph $X_{g}$
is not bipartite, and its girth is at least $\frac{2}{3}\log_{q}\left(\left|X_{g}\right|\right)$.
\end{enumerate}
\end{thm}

We now consider the $q+1$ elements $S'=\left\{ \delta_{1},...,\delta_{q+1}\right\} \subset A^{\times}\left(\F_{q}\left(u\right)\right)/Z^{\times}$
satisfying $\delta_{k}=\gamma_{k}^{2}$, so explicitly
\[
\delta_{k}=2-u+2c_{k}j+2d_{k}ij.
\]
Let $\Gamma'=\left\langle \delta_{1},...,\delta_{q+1}\right\rangle \le\Gamma$.

To prove Theorem~\ref{thm:vertex expansion intro} we need to understand
the group $\Gamma'/\left(\Gamma'\cap\Gamma\left(g\right)\right)$
and its Cayley structure relative to the generators $\delta_{1},...,\delta_{q+1}$.

It is simpler to work with valuations instead of divisions. Recall
that the valuations of the field $\F_{q}\left(u\right)$ are $v_{1/u}$
defined by $v_{1/u}\left(\frac{f}{g}\right)=\deg_{u}g-\deg_{u}f$,
$f,g\in\F_{q}\left[u\right]$, and for every irreducible monic polynomial
$p\in\F_{q}\left(u\right)$ the valuation $\nu_{p}\left(p^{a}\frac{f}{g}\right)=a$,
where $f,g\in\F_{q}\left(u\right)$ are not divisible by $p$.

Using the language of valuations, $\Gamma\left(g\right)$ contains
all the elements of $A^{\times}\left(\F_{q}\left(u\right)\right)/Z^{\times}$,
which are in the free group generated by $\gamma_{1},...,\gamma_{q+1}$,
and further have an element $\alpha=a+bi+cj+dij$ in their equivalence
class satisfying:

\[
v_{g}\left(a\right)=0,v_{g}\left(b\right)>0,v_{g}\left(c\right)>0,v_{g}\left(d\right)>0.
\]

Next we make the change of variables $t=\frac{u}{2-u}$. It holds
that $u=\frac{2t}{t+1}$, $u-1=\frac{t-1}{t+1}$ and $2-u=\frac{2}{t+1}$.

When make a change of variables, the quaternion algebra $A$ changes to the quaternion algebra $A_1=\operatorname{span}\{ 1,i_1,j_1,i_1j_1\} $
over $\F_q(t)$ with $i_1^2=\epsilon$, $j_1^2=\frac{t-1}{t+1}$,
$i_1j_1=-j_1 i_1$. In the new algebra, we have
\begin{align*}
\gamma_k & =1+c_k j_1 +d_k i_1 j_1 \\
\delta_k & =\frac{2}{t+1}+2c_k j_1+2d_k i_1 j_1.
\end{align*}

Let $T\colon \F_{q}\left(u\right)\to\F_{q}\left(t\right)$, $T\left(f\left(u\right)\right)=f\left(\frac{2t}{t+1}\right)$
be the isomorphism of fields defined by the change of coordinates.
There is a bijection between valuations $v$ of $\F_{q}\left(u\right)$
and valuations $\sigma$ of $\F_{q}\left(t\right)$, defined by $v\left(f\right)=\sigma\left(T\left(f\right)\right)$
for every $f\in\F_{q}\left(u\right)$. Let us describe this bijection.

The change of variables is a composition of two simpler operations: A linear transformation $t=au+b$, $a\ne0,b\in\F_{q}$
and an inversion $t=1/u$.

For $t=au+b$, the bijection is as follows: $v_{1/u}$ corresponds
to $\sigma_{1/t}$, and for $g\left(u\right)$ monic irreducible,
$\deg g=m'$, let $h\left(t\right)=a^{-m'}g\left(at+b\right)$.
Then $v_{g}$ corresponds to $\sigma_{h}$.

For $t=1/u$, the bijection is as follows: $v_{1/u}$ corresponds
to $\sigma_{t}$, $v_{u}$ corresponds to $\sigma_{1/t}$, and for
$g\left(u\right)$ monic irreducible, $g\left(u\right)\ne u$, $\deg g=m'$,
$g(u)=u^{m'}+a_{m'-1}u+....+a_{0}$, let $h\left(t\right)=a_{0}^{-1}t^{m'}g\left(1/t\right)=t^{m'}+a_{1}a_{0}^{-1}t^{m'-1}+...+a_{m'-1}a_{0}^{-1}t+a_{0}^{-1}$.
Then $v_{g}$ corresponds to $\sigma_{h}$.

Applying the above to $t=\frac{u}{2-u}$, $u=\frac{2t}{t+1}$, the
bijection is
\begin{align*}
v_{1/u} & \leftrightarrow\sigma_{t+1}\\
v_{u-2} & \leftrightarrow\sigma_{1/t},
\end{align*}
and for $g\left(u\right)\ne u-2$ of degree $m'$, let $h\left(t\right)$
the monic polynomial corresponding to $\left(t+1\right)^{m'}g\left(\frac{2t}{t+1}\right)$.
Then $v_{g}\leftrightarrow\sigma_{h}$.

The other direction of this correspondence is given as follows: For
$h\left(t\right)\ne t+1$ of degree $m'$, let $g\left(t\right)$
be the monic polynomial corresponding to $\left(u-2\right)^{m'}h\left(\frac{u}{2-u}\right)$.
Then $\sigma_{h}\leftrightarrow v_{g}$.
\begin{lem}
Using the correspondence above, for $g\left(u\right)\ne u-2$, it
holds that $\left(\frac{u}{g\left(u\right)}\right)=\left(\frac{2t\left(t+1\right)}{h\left(t\right)}\right)$.
\end{lem}

\begin{proof}
The Legendre symbol $\left(\frac{f\left(u\right)}{g\left(u\right)}\right)$
for $v_{g}\left(f\right)\ge 0$ is determined by whether the image of
$f$ in the finite field $\left\{ f'\in\F_{q}\left(u\right): v_{g}\left(f'\right)\ge0\right\} /\left\{ f'\in\F_{q}\left(u\right):v_{g}\left(f'\right)>0\right\} $
is zero, a non-zero square, or neither. Since $u=\frac{2t}{t+1}$ and $2t\left(t+1\right)=\frac{2t}{t+1}\left(t+1\right)^{2}$, the result follows.
\end{proof}
Returning to $\Gamma\left(g\right)$, let $h\left(t\right)$ correspond
to $g\left(u\right)$ as above. Then after the change of coordinates we have:
\begin{lem}
The group $\Gamma\left(g\right)$ is isomorphic subgroup of $A_{1}^{\times}\left(\F_{q}\left(t\right)\right)/Z^{\times}$
generated by $\gamma_{1},..,\gamma_{q+1}$, $\gamma_{k}=1+c_{k}j_{1}+d_{k}i_{1}j_{1}$,
such that there is an element in the equivalent class satisfying

\begin{equation}
v_{h}\left(a\right)=0,v_{h}\left(b\right)>0,v_{h}\left(c\right)>0,v_{h}\left(d\right)>0.\label{eq:Conditions}
\end{equation}
The group $\Gamma/\Gamma\left(g\right)$ is isomorphic to $\PGL_{2}\left(\F_{q^{2m}}\right)$
if and only if $\left(\frac{2t\left(t+1\right)}{h\left(t\right)}\right)=-1$.
\end{lem}

Next, we change the algebra to an equivalent algebra. Let $A_{2}=\operatorname{span}\left\{ 1,i_{2},j_{2},i_{2}j_{2}\right\} $
be the quaternion algebra over $\F_{q}\left(t\right)$ with $i_{2}^{2}=\epsilon$,
$j_{2}^{2}=\left(t+1\right)\left(t-1\right)=t^{2}-1$, $i_{2}j_{2}=-j_{2}i_{2}$.
Then $A_{1}\left(\F_{q}\left(t\right)\right)\cong A_{2}\left(\F_{q}\left(t\right)\right)$,
with the explicit isomorphism
\[
a+bi_{1}+cj_{1}+di_{1}j_{1}\in A_{1}\left(\F_{q}\left(t\right)\right)\leftrightarrow a+bi_{2}+\frac{c}{t+1}j_{2}+\frac{d}{t+1}i_{2}j_{2}\in A_{2}\left(\F_{q}\left(t\right)\right).
\]

In the new algebra,
\begin{align*}
\gamma_{k} & =1+\frac{c_{k}}{t+1}j_{2}+\frac{d_{k}}{t+1}i_{2}j_{2}\\
\delta_{k} & =\frac{2}{t+1}+\frac{2}{t+1}c_{k}j_{2}+\frac{2}{t+1}d_{k}i_{2}j_{2}\\
 & =1+c_{k}j_{2}+d_{k}i_{2}j_{2}.
\end{align*}
The last equality follows from the fact that we work in $A_{2}^{\times}\left(\F_{q}\left(t\right)\right)/Z^{\times}$.

Moving to $A_{2}^{\times}\left(\F_{q}\left(t\right)\right)/Z^{\times}$,
$\Gamma$ is generated by $\gamma_{1},...,\gamma_{q+1}\in A_{2}^{\times}\left(\F_{q}\left(t\right)\right)/Z^{\times}$,
while $\Gamma\left(g\right)$ consists of the elements with an element
in their equivalence class satisfying Equation~\eqref{eq:Conditions}.

Therefore $\Gamma'$ is generated by $\delta_{1},...,\delta_{q+1}\in A_{2}^{\times}\left(\F_{q}\left(t\right)\right)$,
and $\Gamma'\cap\Gamma\left(g\right)$ are the elements in $\Gamma'$
satisfying Equation~\eqref{eq:Conditions}.

Denote $s=t^{2}$. Notice that $A_{2}$ is defined over $\F_{q}\left(s\right)$,
and moreover $A_{2}\left(\F_{q}\left(s\right)\right)$ is the quaternion
algebra defined by the relations $i_{2}^{2}=\epsilon$, $j_{2}^{2}=s-1$,
$i_{2}j_{2}=-j_{2}i_{2}$. Therefore, $A_{2}\left(\F_{q}\left(s\right)\right)$
is the same algebra as $A\left(\F_{q}\left(u\right)\right)$, up to changing $u$ to $s$. Moreover,
the elements $\delta_{1},...,\delta_{q+1}$ are actually defined over
$\F_{q}\left(s\right)$, and correspond to the elements $\gamma_{1},...,\gamma_{q+1}$
of $A\left(\F_{q}\left(u\right)\right)$! This is the ``miracle''
underlying this construction.
\begin{rem}\label{rem:explanation of explicit}
Let us stop the proof for a moment and explain the connection between Theorem~\ref{thm:vertex expansion intro} and Theorem~\ref{thm:Abstract theorem}.

After the change of variables, look at the group $G = A_{2}^{\times}\left(\F_{q}((t))\right)/Z^{\times} \cong \PGL_2\left(\F_q((t))\right)$. Morgenstern shows that $\Gamma$ acts simply transitively on the Bruhat-Tits building $B_G$ of $G$, so the Cayley graph of $\Gamma$ with respect to $\gamma_1,...,\gamma_{q+1}$ can be identified with $B_G$.

Denote $H = A_{2}^{\times}\left(\F_{q}\left(\left( t^2 \right)\right)\right)/Z^{\times} \cong \PGL_2 \left(\F_q\left(\left(t^2\right)\right)\right)$ which is a closed subgroup of $G$.
It is not hard to see that $\Gamma' = \Gamma \cap H$.
The calculation above implies that $\Gamma'$ acts simply transitively on the Bruhat-Tits building $B_H$ of $H$, and its Cayley graph with respect to $\delta_1,...,\delta_{q+1}$ is isomorphic to $B_H$. Therefore, $\Gamma' = \Gamma \cap H$ is a lattice in $H$, which is far from obvious. As we explain in Section~\ref{sec:Division Algberas}, this fact also follows from the fact that $A_2$ is actually defined over $\F_q(s)=\F_q\left(t^2\right)$.

The embedding of the group $\Gamma'$ in $\Gamma$ is the same as the embedding of $B_H$ in $B_G$. 
This implies that the embedding of $\Gamma'/ \left(\Gamma'\cap\Gamma(g)\right)$ in $\Gamma/\Gamma(g)$ is the same as the embedding of $\left(\Gamma'\cap\Gamma(g)\right) \backslash B_H$ in $\Gamma(g) \backslash B_G$. 

In the next part of the proof we study the growth of $\left[\Gamma':\left(\Gamma' \cap \Gamma(g)\right)\right]$ relative to $[\Gamma:\Gamma(g)]$ as in Lemma~\ref{lem:small volume in quotient} or Theorem~\ref{thm:Abstract theorem}. We will actually understand a bit more than that.

\end{rem}

Continuing the proof, we may assume that $\Gamma'\subset A_{2}^{\times}\left(\F_{q}\left(s\right)\right)/Z^{\times}$.
However, we still need to handle the conditions of Equation~\eqref{eq:Conditions} to understand $\Gamma'\cap\Gamma\left(g\right)$.

There are two cases: the ``good case'' $h\in\F_{q}\left(t^{2}\right)=\F_{q}\left(s\right)$ (i.e., $h$ only has $t$ to an even power)
and the ``bad case'' $h\notin\F_{q}\left(t^{2}\right)=\F_{q}\left(s\right)$.

In the good case, let $\tilde{h}\left(s\right)\in\F_{q}\left(s\right)$
be polynomial satisfying $\tilde{h}\left(t^{2}\right)=h\left(t\right)$.
Notice that $\deg\left(\tilde{h}\right)=\deg\left(h\right)/2=\deg\left(g\right)/2=m$.

In the bad case, let $\tilde{h}\left(s\right)\in\F_{q}\left(s\right)$
be the polynomial satisfying $\tilde{h}\left(t^{2}\right)=h\left(t\right)h\left(-t\right)$.

In both cases, $\tilde{h}\left(s\right)\in\F_{q}\left(s\right)=\F_{q}\left(t^{2}\right)$
is an irreducible polynomial (or prime), which lies below the irreducible
polynomial $h\left(t\right)\in\F_{q}\left(t\right)$ in the extension
of $\F_{q}\left(s\right)$ to $\F_{q}\left(t\right)$. In other words,
it holds that $h\left(t\right)\F_{q}\left[t\right]\cap\F_{q}\left[s\right]=\tilde{h}\left(s\right)\F_{q}\left[s\right]$
and for $f\in\F_{q}\left(s\right)$, $v_{h}\left(f\left(t^{2}\right)\right)=v_{\tilde{h}}\left(f\left(s\right)\right)$.

This implies that we may identify $\Gamma'$ as the subgroup of $A_{2}^{\times}\left(\F_{q}\left(s\right)\right)/Z^{\times}$,
generated by $\delta_{k}=1+c_{k}j_{2}+d_{k}i_{2}j_{2}$, and $\Gamma'\cap\Gamma\left(g\right)$
as its subgroup of elements with an element in the equivalence class
satisfying $v_{\tilde{h}}\left(a\right)=0,v_{\tilde{h}}\left(b\right)>0,v_{\tilde{h}}\left(c\right)>0,v_{\tilde{h}}\left(d\right)>0$.

The final description is exactly the description of the Morgenstern
graph, with $u$ replaced by $s$ and $g\left(u\right)$ replaced
by $\tilde{h}\left(s\right)$. Therefore:
\begin{thm}
$Y_{g}=\Cayley\left(\Gamma'/\left(\Gamma'\cap\Gamma\left(g\right)\right),\left\{ \delta_{1},...,\delta_{q+1}\right\} \right)$
is isomorphic to the Morgenstern graph $X_{\tilde{h}}$.

In particular, the subgroup $\Gamma'/\left(\Gamma'\cap\Gamma\left(g\right)\right) \le \Gamma / \Gamma (g)$ is isomorphic to either $\PSL_2\left(q^{\deg \tilde h}\right)$ or $\PGL_2\left(q^{\deg \tilde h}\right)$.
\end{thm}

Our next goal is to understand which of the two cases, $\PSL_{2}$ or $\PGL_{2}$ happens. In the ``good case'', $\tilde{h}\left(s\right)\in\F_{q}\left(s\right)=\F_{q}\left(t^{2}\right)$
remains irreducible in the extension to $\F_{q}\left(t\right)$. In
other words, it is inert in the extension. Since this is a quadratic
extension, it is well known that it happens if and only if $\left(\frac{s}{\tilde{h}\left(s\right)}\right)=-1$.
In this case, by Theorem~\ref{thm:Morgenstern Theorem}, $Y_{g}$
is a bipartite Cayley graph on $\PGL_{2}\left(\F_{q^{\deg\tilde{h}}}\right)=\PGL_{2}\left(\F_{q^{m}}\right)$.

In the ``bad case'', $h\left(s\right)\in\F_{q}\left(s\right)=\F_{q}\left(t^{2}\right)$
splits in the extension to $\F_{q}\left(t\right)$. This happens if
and only if $\left(\frac{s}{\tilde{h}\left(s\right)}\right)=1$. In
this case, by Theorem~\ref{thm:Morgenstern Theorem}, $Y_{g}$ is
a non-bipartite Cayley graph on $\PSL_{2}\left(\F_{q^{\deg\tilde{h}}}\right)=\PSL_{2}\left(\F_{q^{2m}}\right)$.

For Theorem~\ref{thm:vertex expansion intro}, we need to prove that the good case may happen, and to understand
$X_{g}$ in this case. For this we notice that we may first choose
$\tilde{h}\in\F_{q}\left(s\right)$ of degree $m$, which is inert
in the field extension to $\F_{q}\left(t\right)$, then get $h\left(t\right)=\tilde{h}\left(t^{2}\right)$
and finally get $g\left(u\right)$ from it as the irreducible monic
corresponding to $\left(u-2\right)^{2m}h\left(\frac{u}{2-u}\right)$.

We recall quadratic reciprocity in $\F_{q}\left(s\right)$ (\cite[Theorem 3.3]{rosen2013number}),
which states that for $f,g\in\F_{q}\left[s\right]$ irreducible
\[
\left(\frac{f}{g}\right)=\left(-1\right)^{\frac{q-1}{2}\deg f\deg g}\left(\frac{g}{f}\right).
\]

Then
\[
\left(\frac{s}{\tilde{h}\left(s\right)}\right)=\left(-1\right)^{\frac{q-1}{2}\deg\tilde{h}}\left(\frac{\tilde{h}\left(s\right)}{s}\right)=\left(-1\right)^{\frac{q-1}{2}\deg\tilde{h}}\left(\frac{\tilde{h}\left(0\right)}{s}\right)=\left(-1\right)^{\frac{q-1}{2}\deg\tilde{h}}\left(\frac{\tilde{h}\left(0\right)}{q}\right).
\]

The last element is the usual Legendre symbol in $\Z$.

Next,
\[
\left(\frac{2t\left(t+1\right)}{h\left(t\right)}\right)=\left(\frac{2}{h\left(t\right)}\right)\left(\frac{t}{h\left(t\right)}\right)\left(\frac{t+1}{h\left(t\right)}\right).
\]

Since $h\left(t\right)$ is of even degree, $\F_{q}\left(t\right)/h\left(t\right)\F_{q}\left(t\right)$
contains $\F_{q^{2}}$. Therefore, every $a\in\F_{q}$ has a square root in it and $\left(\frac{2}{h\left(t\right)}\right)=1$. It holds by quadratic reciprocity, since the degree of $h\left(t\right)$ is even, that
\begin{align*}
\left(\frac{t}{h\left(t\right)}\right) & =\left(\frac{h\left(t\right)}{t}\right)=\left(\frac{h\left(0\right)}{t}\right)=\left(\frac{h\left(0\right)}{q}\right)=\left(\frac{\tilde{h}\left(0\right)}{q}\right)\\
\left(\frac{t+1}{h\left(t\right)}\right) & =\left(\frac{h\left(t\right)}{t+1}\right)=\left(\frac{h\left(-1\right)}{t+1}\right)=\left(\frac{h\left(-1\right)}{q}\right)=\left(\frac{\tilde{h}\left(1\right)}{q}\right).
\end{align*}
The last two elements in each row are the Legendre symbol in $\Z$.
Therefore, $\left(\frac{2t\left(t+1\right)}{h\left(t\right)}\right)=\left(\frac{\tilde{h}\left(0\right)}{q}\right)\left(\frac{\tilde{h}\left(1\right)}{q}\right)$
determines whether $X_{g}$ is $\PGL$ or $\PSL$.

We conclude that by determining $\tilde{h}\left(0\right),\tilde{h}\left(1\right)$
we can make $Y_{g}$ be isomorphic to $\PGL_{2}\left(\F_{q^{m}}\right)$
and $X_{g}$ to be isomorphic to either of $\PGL_{2}\left(\F_{q^{2m}}\right)$
or $\PSL_{2}\left(\F_{q^{2m}}\right)$. Finally, for $m$ large enough,
we may freely choose $\tilde{h}\left(0\right),\tilde{h}\left(1\right)$ while keeping $\tilde{h}$ irreducible
by Chabotarev's density theorem (\cite[Theorem 4.7 and Theorem 4.8]{rosen2013number}).

We collect our findings in the following lemma:
\begin{lem}\label{lem:explicit construction cor}
For every $m$ large enough, we may find a monic irreducible polynomial $g$ of degree $2m$ such that:
\begin{enumerate}
    \item The Morgenstern Ramanujan Cayley graph $X_{g}=\Cayley\left(\Gamma/\Gamma\left(g\right),\left\{ \gamma_{1},...,\gamma_{q+1}\right\} \right)$ is bipartite, $\Gamma/\Gamma\left(g\right) \cong \PGL_2\left(\F_{q^{2m}}\right)$ and its girth is greater than $4/3 \log_q\left(|X_g|\right)$.
    \item The subgroup $\Gamma'/\left(\Gamma'\cap\Gamma\left(g\right)\right) \le \Gamma / \Gamma (g)$ that is generated by $\left\{ \gamma_{1}^2,...,\gamma_{q+1}^2\right\}$ is isomorphic to $\PGL_2\left(\F_{q^{m}}\right)$. Moreover, 
    $Y_g = \Cayley\left(\Gamma'/\left(\Gamma'\cap\Gamma\left(g\right)\right), \left\{ \gamma_{1}^2,...,\gamma_{q+1}^2\right\}\right)$ is also a Morgenstern Ramanujan graph.
\end{enumerate}
\end{lem}

Notice that the lemma implies that $|Y_g|=\left|\PGL_2\left(\F_{q^m}\right)\right| = O\left(\sqrt{\left|\PGL_2\left(\F_{q^{2m}}\right)\right|}\right)=O(\sqrt{|X_g|})$. By the discussion at the beginning of this section, we conclude that Lemma~\ref{lem:explicit construction cor} is an explicit version of Lemma~\ref{lem:small volume in quotient}.

We need however another result to complete the explicit part of the proof of  Theorem~\ref{thm:vertex expansion intro}. 
\begin{lem}\label{lem:Symmetry lemma}
Every vertex $x\in N(Y_g)$ is connected to exactly two vertices in $Y_g$. 
\end{lem}
\begin{proof}
If $x\in N(Y_g)$, then $x = \gamma_i y$, for $y\in Y_g$, and $\gamma_i$ a generator. Therefore, $x$ is connected in $X_g$ to both $y = \gamma_i^{-1} x$ and $\gamma_i^2y = \gamma_i x \in Y_g$. Therefore, every neighbor of $Y_g$ is connected to at least two vertices of it (we discussed this part of the proof at the beginning of the section). 

Assume by contradiction that $x\in N(Y_g)$ is connected to more than two vertices of $Y_g$, then by applying the automorphism of the Cayley graph $X_g$ defined by subgroup $\Gamma'/\left(\Gamma'\cap\Gamma\left(g\right)\right)$, we would get $O\left(|Y_g|\right)$ other neighbors of $Y_g$ which are connected to more than 2 vertices of $Y_g$. Then there is some $\delta>0$ such that on average a neighbor of $Y_g$ is connected to $2+\delta$ vertices in $Y_g$. This contradicts Kahale's vertex expansion Theorem~\ref{thm:Kahale vertex expansion}.

\end{proof}

\begin{rem}
Assume that we choose $\tilde{h}\left(s\right)$, $\deg\tilde{h}\left(s\right)=m$,
such that it splits in the extension to $\tilde{h}\left(t^{2}\right)=h\left(t\right)h\left(-t\right)$.
Then a similar construction still works -- we look at $g\left(u\right)=g_{1}\left(u\right)g_{2}\left(u\right)$,
where $g_{1}\left(u\right)$ corresponds to $h\left(t\right)$ and
$g_{2}\left(u\right)$ corresponds to $h\left(-t\right)$. Then $\Gamma\left(g\right)=\Gamma\left(g_{1}\right)\cap\Gamma\left(g_{2}\right)$
defines a Cayley graph $X_{g}=\Cayley\left(\Gamma/\Gamma\left(g\right),\left\{ \gamma_{1},...,\gamma_{q+1}\right\} \right)$.
There is an embedding $F\colon Y_{g}\cong X_{\tilde{h}}\to X_{g}$, which
extends to a graph map on the subdivision graph $Y_{g}^{\prime}$
of $Y_{g}$.

In this case $Y_{g}$ will be a Cayley graph on $\PSL_{2}\left(\F_{q^{m}}\right)$,
while $X_{g}$ will be a Cayley graph on some subgroup between $\PSL_{2}\left(\F_{q^{m}}\right)\times\PSL_{2}\left(\F_{q^{m}}\right)$
and $\PGL_{2}\left(\F_{q^{m}}\right)\times\PGL_{2}\left(\F_{q^{m}}\right)$.
So again we get a similar map from a small $\left(q+1,2\right)$-biregular
graph and a big $\left(q+1\right)$-regular Ramanujan graph.
\end{rem}

\subsection{Explicit Generators}

We shortly describe how to construct our graphs explicitly. Assume we are given monic irreducible $\tilde{h}\left(s\right)\in\F_{q}\left(s\right)$,
$\deg\tilde{h}=m$, which is inert (remains irreducible) in the extension to $\F_{q}\left(t\right)$.

Let $h\left(t\right)=\tilde{h}\left(t^{2}\right)$ be the irreducible
polynomial of degree $2m$ in $\F_{q}\left(t\right)$ above $\tilde{h}\left(s\right)$.
Let $\epsilon\in\F_{q}$ be a non-square and let $\boldsymbol{i}\in\F_{q}\left(t\right)/h\left(t\right)\F_{q}\left(t\right)$
be a square root of $\epsilon$ (which exists since $h$ is of even
degree). Consider the following elements in $\PGL_{2}\left(\F_{q}\left(t\right)/h\left(t\right)\F_{q}\left(t\right)\right)\cong\PGL_{2}\left(\F_{q^{2m}}\right)$:
\begin{align*}
\gamma_{k} & =\left(\begin{array}{cc}
t+1 & \left(c_{k}-d_{k}\boldsymbol{i}\right)\\
\left(c_{k}+d_{k}\boldsymbol{i}\right)\left(t^{2}-1\right) & t+1
\end{array}\right)\\
\delta_{k} & =\gamma_{k}^{2}=\left(\begin{array}{cc}
1 & \left(c_{k}-d_{k}\boldsymbol{i}\right)\\
\left(c_{k}+d_{k}\boldsymbol{i}\right)\left(t^{2}-1\right) & 1
\end{array}\right)
\end{align*}
(we work modulu center, so it makes sense). Recall that $\left(c_{k},d_{k}\right)$
are the $\left(q+1\right)$ solutions to $\epsilon d^{2}-c^{2}=1$.

Then $\gamma_{1},...,\gamma_{q+1}$ generate a Ramanujan Cayley graph
isomorphic to the Morgenstern Cayley graph of the monic polynomial
corresponding to $\left(u-2\right)^{2m}h\left(\frac{u}{2-u}\right)$.
The elements $\left\{ \gamma_{1},...,\gamma_{q+1}\right\} $ generate
$\PGL_{2}\left(\F_{q^{2m}}\right)$ if and only if $\left(\frac{2t\left(t+1\right)}{h\left(t\right)}\right)=\left(\frac{\tilde{h}\left(0\right)}{q}\right)\left(\frac{\tilde{h}\left(1\right)}{q}\right)=-1$.
The elements $\delta_{k}$ generate a Ramanujan Cayley graph isomorphic to
$\PGL_{2}\left(\F_{q^{m}}\right)$. This is simplest to see when $m$
itself is even, since then $\epsilon\in\F_{q}\left(t^{2}\right)/\tilde{h}\left(t^{2}\right)\F_{q}\left(t^{2}\right)\subset\F_{q}\left(t\right)/h\left(t\right)\F_{q}\left(t\right)$,
and the $\delta_{k}$ are the generators given in \cite[Equation (14)]{morgenstern1994existence}.

\section{\label{sec:The-Method}The Closed Orbit Method}

The goal of this section is to formalize and prove Theorem~\ref{thm:Abstract theorem},
which we do in Theorem~\ref{thm:Good Pair Theorem}. The proof follows the standard construction of arithmetic cocompact lattices,
and keeps track of its behavior when taking field extensions. The essential results are the second part of Theorem~\ref{thm:choice of Gamma} and Lemma~\ref{lem:intersection of congreunce}, which relates lattices in a group with a lattice in a subgroup.

We use several standard results about semisimple groups over global fields, which
may be found in Prasad's work \cite{prasad1977strong}. To keep our
arguments short, we assume familiarity with the theory.

Here is our general setting: Let $k$ be a global field (a number field or a finite extension of $\F_{q}\left(t\right)$) and let $l$ be
a finite separable field extension of $k$. We denote the places of
$k$ using the letter $v$ with $k_{v}$ being the field completion.
Similarly, we denote the places of $l$ by the letter $w$, with $l_{w}$
being the field completion. For the non-Archimedean places, we let
$\O_{v}$ (resp. $\O_{w}$) be the ring of integers of $k_{v}$ (resp.
$l_{w}$), and let $\pi_{v}$ (resp. $\pi_{w}$) be a uniformizer
of $k_{v}$ (resp. $l_{w}$).

Let $\A_{k}$ and $\A_{l}$ be the adele rings of $k$ and $l$, i.e.,
\[
\A_{k}=\left\{ \left(x_{v}\right)\in\prod_{v}k_{v}:x_{v}\in\O_{v}\text{ for almost every \ensuremath{v}}\right\} ,
\]
and similarly for $\A_{l}$.

Let $\boldsymbol{G}$ be a semisimple algebraic group defined over
$k$, defined by set of equations as a subgroup of $\SL_{N}$. We
will use the same notation $\boldsymbol{G}$ for its extension of
scalars to $l$. We assume that $\boldsymbol{G}$ is connected, simply
connected, and almost-simple over $k$ and over $l$. Let $\boldsymbol{\tilde{G}}$
be the split form of $\boldsymbol{G}$, so over the algebraic closure
$\overline{k}$ of $k$ (and $l$), $\boldsymbol{G}\left(\overline{k}\right)\cong\boldsymbol{\tilde{G}}\left(\overline{k}\right)$.

Notice that $\boldsymbol{G}\left(k_{v}\right)$ and $\boldsymbol{G}\left(l_{w}\right)$
are well-defined, and have a natural topology coming from the embedding
$\boldsymbol{G}\left(k_{v}\right)\subset\SL_{N}\left(k_{v}\right)\subset M_{N}\left(k_{v}\right)$.
In addition, $\boldsymbol{G}\left(\O_{v}\right)$ and $\boldsymbol{G}\left(\O_{w}\right)$
are also well-defined for almost every $v$ and almost every $w$.
For almost every place $v$ of $k$, the isomorphism $\boldsymbol{G}\left(\overline{k}\right)\cong\boldsymbol{\tilde{G}}\left(\overline{k}\right)$
is actually defined over $k_{v}$, so there is an isomorphism $\boldsymbol{G}\left(k_{v}\right)\cong\boldsymbol{\tilde{\boldsymbol{G}}}\left(k_{v}\right)$.
Moreover, after dropping a finite number of places $v$ of $k$, $\boldsymbol{G}\left(\O_{v}\right)\cong\boldsymbol{\tilde{G}}\left(\O_{v}\right)$
is well-defined, and is a maximal compact open subgroup of $\boldsymbol{G}\left(k_{v}\right)$.
For such $v$, let $w_{1},...,w_{m}$ be the places of $l$ over $v$.
Then $\boldsymbol{G}\left(l_{w_{i}}\right)\cong\boldsymbol{\tilde{G}}\left(l_{w_{i}}\right)$,
$\boldsymbol{G}\left(\O_{w_{i}}\right)\cong\boldsymbol{\tilde{G}}\left(\O_{w_{i}}\right)$
and there is a diagonal embedding $\boldsymbol{G}\left(k_{v}\right)\to\prod_{i=1}^{m}\boldsymbol{G}\left(k_{w_{i}}\right)$,
such that $\boldsymbol{G}\left(k_{v}\right)\cap\prod\boldsymbol{G}\left(\O_{w_{i}}\right)=\boldsymbol{G}\left(\O_{v}\right)$.

For the finite number of non-Archimedean places where the above
does not hold, the embedding $\boldsymbol{G}\left(k_{v}\right)\to\prod_{i=1}^{m}\boldsymbol{G}\left(k_{w_{i}}\right)$
is still well-defined, and for every compact open subgroup of $\prod_{i=1}^{m}K_{w_{i}}\le\prod_{i=1}^{m}\boldsymbol{G}\left(l_{w_{i}}\right)$,
the subgroup $K_{v}=\prod_{i=1}^{m}K_{w_{i}}\cap\boldsymbol{G}\left(k_{v}\right)\le\boldsymbol{G}\left(k_{v}\right)$
is a compact open subgroup.

By the above, the adele group
\[
\boldsymbol{G}\left(\A_{k}\right)=\left\{ \left(g_{v}\right)\in\prod_{v}\boldsymbol{G}\left(k_{v}\right):g_{v}\in\boldsymbol{G}\left(\O_{v}\right)\text{ for almost every \ensuremath{v}}\right\}
\]
is well-defined, and has a natural topology coming from the topology
of $\A_{k}$ and the embedding $\boldsymbol{G}\left(\A_{k}\right)\subset\SL_{N}\left(\A_{k}\right)\subset M_{N}\left(\A_{k}\right)$.
We have a diagonal embedding $\boldsymbol{G}\left(k\right)\to\boldsymbol{G}\left(\A_{k}\right)$.
The same is true for $l$.

The local embeddings $\boldsymbol{G}\left(k_{v}\right)\to\prod_{i=1}^{m}\boldsymbol{G}\left(k_{w_{i}}\right)$
extend to a global embedding
$\boldsymbol{G}\left(\A_{k}\right)\to\boldsymbol{G}\left(\A_{l}\right).$
Under this embedding, it holds that
$\boldsymbol{G}\left(\A_{k}\right)\cap\boldsymbol{G}\left(l\right)=\boldsymbol{G}\left(k\right)$.
We consider all of our groups as subgroups of $\boldsymbol{G}\left(\A_{l}\right)$.
For finite sets of places $V$, $W$, there are natural projection
maps
\begin{align*}
P_{V} & :\boldsymbol{G}\left(\A_{k}\right)\to\prod_{v\in V}\boldsymbol{G}\left(k_{v}\right)\\
P_{W} & :\boldsymbol{G}\left(\A_{l}\right)\to\prod_{w\in W}\boldsymbol{G}\left(l_{w}\right).
\end{align*}

Recall that $\boldsymbol{G}$ is called isotropic over a field $F$
if $G\left(F\right)$ contains a non-trivial split $F$-torus, and is called anisotropic over $F$ otherwise. If $F$ is a local field, being anisotropic
is equivalent to $G\left(F\right)$ being compact in the appropriate
topology. We assume that $\boldsymbol{G}$ is anisotropic over $l$,
which implies that it is also anisotropic over $k$.

We have the following two basic theorems:
\begin{thm}[Borel, Behr, Harder]
\label{thm:Gk in Ga is lattice}$\boldsymbol{G}\left(k\right)\le\boldsymbol{G}\left(\A_{k}\right)$
and $\boldsymbol{G}\left(l\right)\le\boldsymbol{G}\left(\A_{l}\right)$
are cocompact lattices.
\end{thm}

\begin{thm}[The strong approximation property -- Platonov, Prasad]
$G$ satisfies the strong approximation property over $k$ and over
$l$, i.e., for every place $v$ where $k$ is isotropic, $\boldsymbol{G}\left(k\right)\boldsymbol{G}\left(k_{v}\right)$
is dense in $\boldsymbol{G}\left(\A_{k}\right)$. The same is true
for $l$.
\end{thm}

Let $V_{0}$ be a finite set of places of $k$. We let $W_{0}$ be the places of $l$ which lie over the places of $V_{0}$. We
assume that $\boldsymbol{G}$ is isotropic over some place in $V_{0}$,
which implies that it is isotropic over some place in $W_{0}$. If
we work over number fields, we further assume that $\boldsymbol{G}$
is anisotropic over all the Archimedean places of $l$ which are not
in $W_{0}$ (and therefore also over all the Archimedean places of
$k$ which are not in $V_{0}$).

We collect our various assumptions into the following definition.
\begin{defn}
\label{def:good pair}Assume that there exist:
\begin{enumerate}
\item A global field $k$ and $l$ a finite separable field extension of
$k$.
\item A semisimple algebraic group $\boldsymbol{G}$ defined over $k$ which
is connected, simply connected, and almost-simple over $k$ and over
$l$. Moreover, $\boldsymbol{G}$ is anisotropic over $l$.
\item A finite set $V_{0}$ of places of $k$, with $W_{0}$ the places
of $l$ over the places of $V_{0}$. The group $\boldsymbol{G}$ is
isotropic over some place in $V_{0}$, and is anisotropic over all
the Archimedean places of $l$ which are not in $W_{0}$.
\end{enumerate}
Then we say that the pair $\left(G,H\right)$ of a group $G$
and subgroup $H\le G$ is good, where

\begin{align*}
G & =\prod_{w\in W_{0}}\boldsymbol{G}\left(l_{w}\right)\\
H & =\prod_{v\in V_{0}}\boldsymbol{G}\left(k_{v}\right).
\end{align*}
\end{defn}

Given a good pair $\left(G,H\right)$, our next goal is to choose
a cocompact lattice $\Gamma\le G$, such that $\Gamma\cap H\le H$
is also a cocompact lattice.

Notice that $G$ embeds into $\boldsymbol{G}\left(\A_{l}\right)$
by
\[
G\cong\prod_{w\in W_{0}}\boldsymbol{G}\left(l_{w}\right)\times\prod_{w\notin W_{0}}\left\{ \operatorname{id}\right\} \le\boldsymbol{G}\left(\A_{l}\right)
\]
and similarly $H$ embeds into $\boldsymbol{G}\left(\A_{k}\right)$.

We let
\begin{align*}
\boldsymbol{G}\left(\A_{l}^{W_{0}}\right) & =\left\{ \left(g_{w}\right)_{w}\in\boldsymbol{G}\left(\A_{l}\right):\forall w\in W_{0},g_{w}=\id\right\} \\
\boldsymbol{G}\left(\A_{k}^{V_{0}}\right) & =\left\{ \left(g_{v}\right)_{v}\in\boldsymbol{G}\left(\A_{k}\right):\forall v\in V_{0},g_{v}=\id\right\} ,
\end{align*}
with the natural induced topology. It holds that $\boldsymbol{G}(\A_l)=G \boldsymbol{G}(\A_l^{W_0})$ and $\boldsymbol{G}(\A_k)=H \boldsymbol{G}(\A_k^{V_0})$.

For $w\not\in W_{0}$ non-Archimedean, fix a compact open subgroup
$K_{w}\le\boldsymbol{G}\left(l_{w}\right)$, equals almost everywhere
to $\boldsymbol{G}\left(\O_{w}\right)$. For the $w\notin W_{0}$
Archimedean, let $K_{w}=\boldsymbol{G}\left(l_{w}\right)$, which
is compact by our assumptions. For $v\notin V_{0}$, we let $K_{v}=\prod_{i=1}^{m}K_{w_{i}}\cap\boldsymbol{G}\left(k_{v}\right)$,
where $w_{1},...,w_{m}$ are the places of $l$ over $v$. By the
above, $K_{v}$ is a compact open subgroup of $\boldsymbol{G}\left(k_{v}\right)$,
which is equal almost everywhere to $\boldsymbol{G}\left(\O_{v}\right)$.

The choices of $K_{w}$ define the compact open subgroup $\boldsymbol{K_{l}}=\prod_{w\in W_{0}}\left\{ \id\right\} \times\prod_{w\not\in W_{0}}K_{w}$
of $\boldsymbol{G}\left(\A_{l}^{W_{0}}\right)$. Then $\boldsymbol{K_{k}}=\boldsymbol{K_{l}}\cap\boldsymbol{G}\left(\A_{k}\right)$
is a compact open subgroup of $\boldsymbol{G}\left(\A_{k}^{V_{0}}\right)$
of the form $\boldsymbol{K_{k}}=\prod_{v\in V_{0}}\left\{ \id\right\} \times\prod_{v\not\in V_{0}}K_{v}$.
Notice that $\boldsymbol{K_{l}}$ and $G$ commute with each other,
and $G\boldsymbol{K_{l}}$ is an open subgroup of $\boldsymbol{G}\left(\A_{l}\right)$.

We choose
\begin{align}
\Gamma & =P_{W_{0}}\left(\boldsymbol{G}\left(l\right)\cap G\boldsymbol{K_{l}}\right)\label{eq:Gamma def}\\
\Gamma' & =P_{V_{0}}\left(\boldsymbol{G}\left(k\right)\cap H\boldsymbol{K_{k}}\right),\nonumber
\end{align}
so that $\Gamma,\Gamma'$ are subgroups of $G=\prod_{w\in W_{0}}\boldsymbol{G}\left(k_{w}\right)$
and $H=\prod_{v\in V_{0}}\boldsymbol{G}\left(k_{v}\right)$.
\begin{thm}
\label{thm:choice of Gamma}
\begin{enumerate}
\item $\Gamma$ (resp. $\Gamma^{\prime}$) is
a cocompact lattice in $G$ (resp. $H$). \item It holds that $\Gamma^{\prime}=H\cap\Gamma$.
\end{enumerate}
\end{thm}

\begin{proof}
Most of this theorem is standard, and we avoid giving the full details.
See, for example, \cite[Lemma 3.1]{morgenstern1994existence} for a similar statement.

The discreteness of $\Gamma$ in $G$ essentially follows from the
discreteness of $l\cap\left(\prod_{w\notin W_{0}}\O_{w}\right)$ in
$\prod_{w\in W_{0}}l_{w}$.

By Theorem~\ref{thm:Gk in Ga is lattice}, $G\left(l\right)$ is
a cocompact lattice in $\boldsymbol{G}\left(\A_{l}\right)$. By the
strong approximation property and our assumptions, $\boldsymbol{G}\left(l\right)G$
is dense in $\boldsymbol{G}\left(\A_{l}\right)$. Moreover, $G\boldsymbol{K}_{l}$
is open in $\boldsymbol{G}\left(\A_{l}\right)$. This implies that
\[
\boldsymbol{G}\left(l\right)G\boldsymbol{K}_{l}=\boldsymbol{G}\left(\A_{l}\right).
\]
Therefore, the map
\[
G\to G\left(l\right)\backslash\boldsymbol{G}\left(\A_{l}\right)/\boldsymbol{K_{l}}
\]
is onto, and $g,g'\in G$ are sent to the same element if there are
$k\in\boldsymbol{K_{l}}$ and $\gamma\in G\left(l\right)$ such that
\[
\gamma g=g'k.
\]
However, $g$ and $k$ commute, so we have $\gamma=g'g^{-1}k$. This
implies that $\gamma\in\boldsymbol{G}\left(l\right)\cap G\boldsymbol{K_{l}}$
and $g'g^{-1}=P_{W_{0}}\left(\gamma\right)$.

On the other hand, if $\gamma\in\boldsymbol{G}\left(l\right)\cap G\boldsymbol{K_{l}}$
and $g'g^{-1}=P_{W_{0}}\left(\gamma\right)$, then we may choose $k\in\boldsymbol{K_{l}}$
such that $\gamma g=g'k$. We conclude that there is an isomorphism
\begin{align*}
\boldsymbol{G}\left(l\right)\backslash\boldsymbol{G}\left(\A_{l}\right)/\boldsymbol{K_{l}} & \cong P_{W_{0}}\left(\boldsymbol{G}\left(l\right)\cap\boldsymbol{K_{l}}\right)\backslash G\\
 & =\Gamma\backslash G.
\end{align*}
By considering the $G$ right action on the two spaces and the compatibility
of measures, we deduce that $\Gamma$ is a cocompact lattice in $G$.

For the same reasons, $\Gamma'=\pi_{V}\left(\boldsymbol{G}\left(k\right)\cap\boldsymbol{\boldsymbol{K}_{k}}\right)\le H$
is cocompact a lattice in $H$.

Finally, and this is where our proof is slightly less standard, since
$P_{W_{0}}\left(\boldsymbol{G}\left(l\right)\right)\cap H=P_{W_{0}}\left(\boldsymbol{G}\left(k\right)\right)\cap H$,
$\boldsymbol{G}\left(k\right)\cap G\boldsymbol{K_{l}}=\boldsymbol{G}\left(k\right)\cap H\boldsymbol{K_{k}}$,
and $P_{W_{0}}\left(\boldsymbol{G}\left(\A_{k}\right)\right)=P_{V_{0}}\left(\boldsymbol{G}\left(\A_{k}\right)\right)$,
\begin{align*}
\text{\ensuremath{\Gamma\cap H}} & =P_{W_{0}}\left(\boldsymbol{G}\left(l\right)\ensuremath{\cap G\boldsymbol{K_{l}}}\right)\cap H\\
 & =P_{W_{0}}\left(\boldsymbol{G}\left(k\right)\cap G\boldsymbol{K_{l}}\right)\cap H\\
 & =P_{W_{0}}\left(\boldsymbol{G}\left(k\right)\cap H\boldsymbol{\boldsymbol{K}_{k}}\right)\\
 & =P_{V_{0}}\left(\boldsymbol{G}\left(k\right)\cap H\boldsymbol{\boldsymbol{K}_{k}}\right)=\Gamma',
\end{align*}
as needed.
\end{proof}
Next, we construct congruence subgroups of $\Gamma$. Let $v_{1}$
be a place of $k$ where it holds that $K_{v_{1}}=\boldsymbol{G}\left(\O_{v}\right)\cong\boldsymbol{\tilde{G}}\left(\O_{v}\right)$.
Moreover, we want $v_{1}$ to be inert in the extension to $l$, and
let $w_{1}$ be the place of $l$ over it. By Chebotarev's density
theorem, there are infinitely many such $v_{1}$.

Notice that we have a $(\operatorname{mod}w_{1})$-homomorphism $\boldsymbol{G}\left(\O_{w_{1}}\right)\to\boldsymbol{G}\left(\O_{w_{1}}/\pi_{w_{1}}\O_{w_{1}}\right)$.
This map is onto, since the group $\boldsymbol{G}\left(\O_{w_{1}}/\pi_{w_{1}}\O_{w_{1}}\right)\cong\boldsymbol{\tilde{G}}\left(\O_{w_{1}}/\pi_{w_{1}}\O_{w_{1}}\right)$
is generated by its unipotent elements, which may be lifted to $\boldsymbol{G}\left(\O_{w_{1}}\right)$.

As $P_{w_{1}}\left(\boldsymbol{G}\left(l\right)\cap G\boldsymbol{K_{l}}\right)\subset\boldsymbol{G}\left(\O_{w_{1}}\right)$,
we have a homomorphism $p_{w_{1}}\colon \Gamma\to\boldsymbol{G}\left(\O_{w_{1}}/\pi_{w_{1}}\O_{w_{1}}\right)$.
Similarly, we have a homomorphism $p_{v_{1}}^{\prime}\colon\Gamma'\to\boldsymbol{G}\left(\O_{v_{1}}/\pi_{v_{1}}\O_{v_{1}}\right)$.

Now, choose the ``congruence subgroups''
\begin{align}
\Gamma\left(w_{1}\right) & =\ker p_{w_{1}}\vartriangleleft\Gamma\label{eq:Gamma(w) def}\\
\Gamma^{\prime}\left(v_{1}\right) & =\ker p_{v_{1}}^{\prime}\vartriangleleft\Gamma'.\nonumber
\end{align}

The split group $\boldsymbol{\tilde{G}}$ is a Chevalley group, which
implies that its size over a finite field $\F$ (over which it is
defined) satisfies:
\[
\boldsymbol{\tilde{G}}\left(\F\right)=\Theta\left(\left|\F\right|^{\dim G}\right).
\]
Since $v_{1}$ is inert in the extension, $\left|\O_{w_{1}}/\pi_{w_{1}}\O_{w_{1}}\right|=\left|\O_{v_{1}}/\pi_{v_{1}}\O_{v_{1}}\right|^{\left[l:k\right]}$.
Therefore, since $\boldsymbol{G}\cong\tilde{\boldsymbol{G}}$ in our
case,
\begin{equation}
\left|\boldsymbol{G}\left(\O_{v_{1}}/\pi_{v_{1}}\O_{v_{1}}\right)\right|=O\left(\left|\boldsymbol{G}\left(\O_{w_{1}}/\pi_{w_{1}}\O_{w_{1}}\right)\right|^{1/\left[l:k\right]}\right).\label{eq:Group Sizes}
\end{equation}

\begin{lem}
The maps $p_{w_1}$ and $p_{v_1}^\prime$ are onto, so
\begin{align*}
\Gamma/\Gamma\left(w_{1}\right) & \cong\boldsymbol{G}\left(\O_{w_{1}}/\pi_{w_{1}}\O_{w_{1}}\right)\\
\Gamma'/\Gamma'\left(v_{1}\right) & \cong\boldsymbol{G}\left(\O_{v_{1}}/\pi_{v_{1}}\O_{v_{1}}\right).
\end{align*}
\end{lem}

\begin{proof}
Let $K_{w_{1}}^{\prime}=\ker\left(\boldsymbol{G}\left(\O_{w_{1}}\right)\to\boldsymbol{G}\left(\O_{w_{1}}/\pi_{w_{1}}\O_{w_{1}}\right)\right)$,
which is a finite index group of $K_{w_{1}}=\boldsymbol{G}\left(\O_{w_{1}}\right)$,
and therefore compact open. Let $\boldsymbol{K_{l}^{\prime}}=\prod_{w\in W_{0}}\left\{ \id\right\} \times K_{w_1}^\prime \times\prod_{w\notin W_{0}\cup\{w_1\}}K_{w}$.
Notice that $\Gamma\left(w_{1}\right)=P_{W_{0}}\left(\boldsymbol{G}\left(l\right)\cap G\boldsymbol{K_{l}^{\prime}}\right)$.

By the strong approximation theorem $\boldsymbol{G}\left(l\right)G\boldsymbol{K_{l}^{\prime}}=\boldsymbol{G}\left(\A_{l}\right)$.
Therefore, for every $k\in K_{w_{1}}$ there is $\gamma\in\boldsymbol{G}\left(l\right)$
and $g\in G$ such that $k\in\gamma g\boldsymbol{K_{l}^{\prime}}$.
This implies that $\gamma\in\boldsymbol{G}\left(l\right)\cap G\boldsymbol{K_{l}}$,
so $P_{W_{0}}\left(\gamma\right)\in\Gamma$. Finally, $p_{w_{1}}\left(P_{W_{0}}\left(\gamma\right)\right)=P_{w_{1}}\left(\gamma\right)=kK_{w_{1}}^{\prime}$,
so $p_{w_1}$ is onto as needed.

The claim for $p_{v_{1}}^{\prime}$ is similar, by choosing
\begin{align*}
K_{v_{1}}^{\prime} & =\ker\left(\boldsymbol{G}\left(\O_{v_{1}}\right)\to\boldsymbol{G}\left(\O_{v_{1}}/\pi_{v_{1}}\O_{v_{1}}\right)\right)\\
 & =K_{v_{1}}\cap K_{w_{1}}^{\prime}
\end{align*}
and $\boldsymbol{K_{k}^{\prime}}=\prod_{v\in V_{0}}\left\{ \id\right\} \times K_{v_{1}}^{\prime}\times\prod_{ v\in V_{0}\cup \{v_0\}}K_{v}$.
\end{proof}
\begin{lem}\label{lem:intersection of congreunce}
$\Gamma\left(w_{1}\right)\cap H=\Gamma'\left(v_{1}\right)$.
\end{lem}

\begin{proof}
Continuing the notations of the last proof $\Gamma\left(w_{1}\right)=P_{W_{0}}\left(\boldsymbol{G}\left(l\right)\cap G\boldsymbol{K_{l}^{\prime}}\right)$,
and
\begin{align*}
\Gamma\left(w_{1}\right)\cap H & =P_{W_{0}}\left(\boldsymbol{G}\left(l\right)\cap G\boldsymbol{K_{l}^{\prime}}\right)\cap H\\
 & =P_{W_{0}}\left(\boldsymbol{G}\left(k\right)\cap G\boldsymbol{K_{l}^{\prime}}\right)\cap H\\
 & =P_{W_{0}}\left(\boldsymbol{G}\left(k\right)\cap H\boldsymbol{K_{k}^{\prime}}\right)\cap H\\
 & =P_{V_{0}}\left(\boldsymbol{G}\left(k\right)\cap H\boldsymbol{K_{k}^{\prime}}\right)\\
 & =\Gamma^{\prime}\left(v_{1}\right).
\end{align*}
\end{proof}
We can finally state the precise form of Theorem~\ref{thm:Abstract theorem}.
\begin{thm}
\label{thm:Good Pair Theorem}Let $\left(G,H\right)$ be a good pair
as above, $\Gamma$ as in Equation~\eqref{eq:Gamma def}, and $\Gamma\left(w_{1}\right)\vartriangleleft\Gamma$
as in Equation~\eqref{eq:Gamma(w) def}. Then $\Gamma\cap H=\Gamma'$
is a cocompact lattice in $H$, and for $\Gamma^{\prime}\left(v_{1}\right)=\Gamma\left(w_{1}\right)\cap H$
it holds that
\[
\left[\Gamma':\Gamma^{\prime}\left(v_{1}\right)\right]=O\left(\left|\Gamma:\Gamma\left(w_{1}\right)\right|^{1/\left[l_{0}:k_{0}\right]}\right).
\]
Therefore, if $W_{0}$ contains only non-Archimedean places and $K$
is a maximal compact open subgroup of $G$ and $K_{H}=H\cap K$ is
the corresponding compact open subgroup of $H$,
\[
\left|\Gamma^{\prime}\left(v_{1}\right)\backslash H/K_{H}\right|=O\left(\left|\Gamma\left(w_{1}\right)\backslash G/K\right|^{1/\left[l_{0}:k_{0}\right]}\right).
\]
\end{thm}

\begin{proof}
Almost all of the theorem was proven above, except for the last part.
It holds that
\[
\left|\Gamma\backslash G/K\right|\left[\Gamma:\Gamma\left(w_{1}\right)\right]\left|\Gamma\cap K\right|^{-1}\le\left|\Gamma\left(w_{1}\right)\backslash G/K\right|\le\left|\Gamma\backslash G/K\right|\left[\Gamma:\Gamma\left(w_{1}\right)\right].
\]
Since $\Gamma\backslash G/K$ and $\Gamma\cap K$ are finite (this
is where we use the fact that $W_{0}$ has no Archimedean places),
\[
\left|\Gamma\left(w_{1}\right)\backslash G/K\right|=\Theta\left(\left[\Gamma:\Gamma\left(w_{1}\right)\right]\right)
\]
The arguments for $\Gamma'$ are the same.

The last part is therefore a consequence of the first part of the theorem.
\end{proof}

\section{\label{sec:Division Algberas}Applications of the Closed Orbit Method}

In this section we explain how the arguments of the last section apply
to groups defined by quaternion algebras and division algebras, which
will then give us interesting lattices in $\SL_{n}$.

\subsection{Closed Orbits in Division Algebras}
Let $A$ be a division algebra over a global field $k$. For simplicity, we assume that $k$ is a function field and the degree $n$ of $A$ is a prime.

Since we assume that $n$ is prime, over any completion $k_{v}$,
$A\left(k_{v}\right)$ is either a division algebra, in which case
we say that $A$ ramifies at $v$ (alternatively, we say that $A_{k_v}$ ramifies), or isomorphic to $M_{n}\left(k_{v}\right)$,
in which case we say that $A$ splits or is unramified at $v$ (alternatively, we say that $A_{k_v}$ splits). The
Albert--Brauer--Hasse--Noether Theorem implies that given any subset of places $v_{1},...,v_{m}$, we may choose $A$ that will ramify only at a subset of $v_{1},...,v_{m}$, and may choose $A\left(k_{v_{1}}\right),...,A\left(k_{v_{m-1}}\right)$
as we want up to isomorphism ($A\left(k_{v_{m}}\right)$ will be determined by the others).

Now let $l$ be a separable field extension of $k$. Given a place $v$ of $k$ and a place $w$ of $l$ over it, it holds that $A\left(l_{w}\right)$
will be split if and only if it is possible to embed $l_{w}$ in $A\left(k_{v}\right)$.
We will restrict ourselves to quadratic extensions, and recall that
$n$ is prime. In this case:
\begin{enumerate}
\item If $n\ne2$ then $A\left(l_{w}\right)$ splits if and only if $A\left(k_{v}\right)$
splits.
\item If $n=2$:
\begin{enumerate}
\item If $A(k_v)$ splits, then $A(l_w)$ splits.
\item If $A(k_v)$ ramifies, then it holds that $A(l_w)$ ramifies if and only if $v$ splits in the extension from $l$ to $k$ (i.e, $w$ is not the only place over $v$).
\end{enumerate}
\end{enumerate}
Now consider the semisimple algebraic group $\boldsymbol{G}$ defined
over $k$ by
\[
\boldsymbol{G}\left(k\right)=\left\{ \alpha\in A\left(k\right):N\left(\alpha\right)=1\right\} .
\]
The group $\boldsymbol{G}$ is connected, simply connected and absolutely
almost simple. The split form of $\boldsymbol{G}$ is $\SL_{n}$.
The group $\boldsymbol{G}$ is anisotropic over a $k_{v}$ (resp.
$l_{w}$) if and only if the algebra $A$ ramifies at $k_{v}$ (resp.
$l_{w}$). The group $\boldsymbol{G}$ is anisotropic over $k$ (resp.
$l$) if and only if $A$ has a ramified place $k_{v}$ (resp. $l_{w}$).

We conclude that Definition~\ref{def:good pair} will hold if \textbf{$\boldsymbol{G}$}
has a ramified place over $l$, and if $V_{0}$ contains a place where
$A$ splits. Let us focus on the case when $V_{0}=\left\{ v_{0}\right\} $
is a split place with a single place $w_{0}$ above $v_{0}$ (so $W_{0}=\left\{ w_{0}\right\} $).
Then $G=\boldsymbol{G}\left(l_{w_{0}}\right)\cong\SL_{n}\left(l_{w_{0}}\right)$
and $H=\boldsymbol{G}\left(k_{v_{0}}\right)\cong\SL_{n}\left(k_{v_{0}}\right)$.
We with to show that Theorem~\ref{thm:Good Pair Theorem} applies
to those cases.
\begin{prop}
\label{prop:good pairs}The pairs
\begin{align*}
 & \left(G,H\right)=\left(\SL_{n}\left(\F_{q}\left(\left(t\right)\right)\right),\SL_{n}\left(\F_{q}\left(\left(t^{2}\right)\right)\right)\right)\\
\text{or\,\,\,} & \left(G,H\right)=\left(\SL_{n}\left(\F_{q^{2}}\left(\left(t\right)\right)\right),\SL_{n}\left(\F_{q}\left(\left(t\right)\right)\right)\right)
\end{align*}
are good, i.e., satisfy Definition~\ref{def:good pair}.
\end{prop}

\begin{proof}
We start with the second case, as it is simpler. We first let $k=\F_{q}\left(t\right)$
and choose the place $v_{0}=t$, so $k_{0}=k_{v_{0}}=\F_{q}\left(\left(t\right)\right)$.
We let $l=\F_{q^{2}}\left(t\right)$ be the quadratic constant field extension of $k$. Then $v_{0}$ is inert in the extension, i.e., have a unique place $w_{0}=t$ over it, with $l_{w_{0}}=l_{0}=\F_{q^{2}}\left(\left(t\right)\right)$.
We then let $A$ be a division algebra of degree $n$ which splits
at $k_{0}$ but ramifies over $k$ and remains ramified over $l$.
It is easy to see that it can happen -- for $n>2$ we may choose
any division algebra over $k$ (it will remain a division algebra
over $l$), and for $n=2$ we need to make sure that $A$ will also
ramify at some place $v$ corresponding to a polynomial of even degree, since $v$ will then split in the extension, and $A$ will ramify for every place $w$ above $v$, so will be ramified over $l$. 
Therefore, all the conditions of Definition~\ref{def:good pair} hold for $G=\boldsymbol{G}\left(l_{0}\right)\cong\SL_{n}\left(\F_{q^{2}}\left(\left(t\right)\right)\right)$,
$H=\boldsymbol{G}\left(k_{0}\right)\cong\SL_{n}\left(\F_{q}\left(\left(t\right)\right)\right)$.

For the first case, we let $k=\F_{q}\left(s\right)$, and let $l$
be a separable quadratic extension of $k$ which ramifies at the place
$v_{0}=s$. For $\mchar q\ne2$ we may take $l=\F_{q}\left(t\right)$
for $t^{2}=s$. For $\mchar q=2$ this extension no longer works
since it is not separable. Instead, we may look at the extension $\F_{q}\left(u\right)$
which we get by adding a root to $u^{2}+su+1=0$. Once again, we choose
a division algebra $A$ such that it splits at $k_{0}$ and ramifies
over $k$ and over $l$. It is simple to see that such $A$ exists.
In any case, if $w_{0}$ is the place of $l$ over $v_{0}$, then
locally we have $l_{0}=\F_{q}\left(\left(t\right)\right)$ for some
$t$ such that $t^{2}=s$. Therefore, $G=\boldsymbol{G}\left(l_{0}\right)\cong\SL_{n}\left(\F_{q}\left(\left(t\right)\right)\right)$
and $H=\boldsymbol{G}\left(k_{0}\right)\cong\SL_{n}\left(\F_{q}\left(\left(t^{2}\right)\right)\right)$,
as needed.
\end{proof}

\subsubsection{The Explicit Construction}\label{subsec:explanation of explicit}

Let us explain how the abstract arguments are related to the explicit
construction. For $\mchar q\ne2$, let $k=\F_{q}\left(s\right)$
and let $l=\F_{q}\left(t\right)$ for $t^{2}=s$. Let $v_{0}=s$,
$w_{0}=t$, so $k_{0}=\F_{q}\left(\left(s\right)\right)$ and $l_{0}=\F_{q}\left(\left(t\right)\right)$.
We choose the algebra $A_{2}$ with the basis $1,i,j,ij$, and the
relations $i^{2}=\epsilon,j^{2}=s-1,ij=-ji$, for $\epsilon\in\F_{q}$
non-square. The algebra $A_{2}$ ramifies at the places $1/s$ and $s-1$. Over $t$ the algebra $A_{2}$ ramifies
at $t-1$ and $t+1$, and is isomorphic to the algebra $A_{1}$ defined
by the relations $i^{2}=\epsilon,j^{2}=\frac{t-1}{t+1},ij=-ji$. After
a change of variables $u=\frac{2t}{2+1}$, the algebra $A_{1}$ is
isomorphic to the algebra $A$ over $\F_{q}\left(u\right)$, defined
by $i^{2}=\epsilon$, $j^{2}=u-1$, $ij=-ji$. Moreover, the place
$t$ corresponds to the place $u$ under this isomorphism. Let $\boldsymbol{G}=\PGL_{1}\left(A_{2}\right)$.

Notice that $\boldsymbol{G}$ is not simply connected, but most of
the general argument of Section~\ref{sec:The-Method} applies to
it, with some modifications we discuss below.

Since $A_{2}$ over $\F_{q}\left(t\right)$ and $A$ over $\F_{q}\left(u\right)$
are isomorphic, the construction of Morgenstern in \cite{morgenstern1994existence}
allows us to find a compact open subgroup $\boldsymbol{K_{l}}\subset\boldsymbol{G}\left(\A_{l}^{\{t\}}\right)$, such that $\Gamma=P_{\left\{ t\right\} }\left(\boldsymbol{G}\left(l\right)\cap G\boldsymbol{K_{l}}\right)\le G=\boldsymbol{G}\left(l_{0}\right)$
is a lattice in $G$ which acts simply transitively on the Bruhat-Tits
building of $G$, and all of its congruence subgroups define Ramanujan
graphs. 

The arguments of Section~\ref{sec:The-Method} (when extended
to the non-simply connected case) say that $\Gamma'=\Gamma\cap H\le H=\boldsymbol{G}\left(k_{0}\right)$
is a lattice in $H$. As a matter of fact, a direct calculation shows that $\Gamma'$ is generated by $\gamma_{1}^{2},...,\gamma_{q+1}^{2}$, where $\gamma_{1},...,\gamma_{q+1}$
are the generators of $\Gamma$, and more importantly, $\Gamma'$
acts simply transitively on the Bruhat-Tits building of $H$. This
implies that the map $F_{\Gamma}\colon\Gamma'\backslash B_{H}\to\Gamma\backslash B_{G}$
is injective, since both sets are of size $1$. Finally, we take congruence
covers $\Gamma\left(w_{1}\right)$ and $\Gamma'\left(v_{1}\right)$
as above, to get a map from a small graph to a large graph.

The fact that $\boldsymbol{G}$ is not simply connected mainly implies
that the exact behavior of $\Gamma/\Gamma\left(w_{1}\right)$ and
$\Gamma'/\Gamma\left(v_{1}\right)$ is not known from the general
arguments. To understand it we need some more work that is done by
Morgenstern in \cite{morgenstern1994existence} and the calculations
of Section~\ref{sec:Proof_of_Explicit_Theorem}.

\subsection{\label{subsec:Bruhat-Tits Buildings}The Bruhat-Tits Building}

In this subsection, let $l_{0}=\F_{q}\left(\left(t\right)\right)$
and we quickly recall the Bruhat-Tits building of $\SL_{n}\left(l_{0}\right)$
and its Ramanujan quotients. When $n=2$, the Bruhat-Tits building
is a tree, and a good reference is \cite{lubotzky1994discrete}. For
$n>2$, see \cite{lubotzky2014ramanujan} and the references therein.

It is simpler to work with $G=\PGL_{n}\left(l_{0}\right)$, as the
buildings are the same, and $\SL_{n}\left(l_{0}\right)$ acts on the
building by its image $\PSL_{n}\left(l_{0}\right)\subset\PGL_{n}\left(l_{0}\right)$ which is of index $n$.
Let $\O=\F_{q}\left[\left[t\right]\right]$ be the ring of integers
of $l_{0}$. Let $K=\PGL_{n}\left(\O\right)$, which is a maximal
compact open subgroup of $G$.

The Bruhat-Tits building $B_{G}$ of $G$ is a clique complex whose
vertices can be identified with $G/K$ (in general, we identify a
building with its vertices). The set $G/K$ can also be described
as all the $\O$ submodules of $l_{0}^{n}$, up to homothety (multiplication by a scalar from $l_{0}$). There is an edge between two modules $\left[M\right]\ne\left[M'\right]$
if they have representatives $M,M'$ such that $tM\subsetneq M'\subsetneq M$.
There is a bijection between equivalence classes of modules $\left[M\right]$
and matrices $A_{M}\in M_{n}\left(\F_{q}\left[t\right]\right)$, of
the form
\[
A_{M}=\left(\begin{array}{cccc}
t^{m_{1}} & f_{1,2}\left(t\right) &  & f_{1,n}\left(t\right)\\
0 & t^{m_{2}} &  & f_{2,n}\left(t\right)\\
\\
0 &  & 0 & t^{m_{n}}
\end{array}\right),
\]
such that:
\begin{enumerate}
\item $m_{1}\ge0,...,m_{n}\ge0$, for $1\le i<j\le n$.
\item $f_{i,j}\in\F_{q}\left[t\right]$ satisfies $\deg f_{i,j}<m_{i}$
for $i<j$.
\item $\gcd\left(t^{m_{1}},...,t^{m_{n}},f_{1,2}\left(t\right),...f_{n-1,n}\left(t\right)\right)=1$.
\end{enumerate}
The bijection is given by sending a matrix $A$ to the module generated by its columns. We identify a module with the corresponding matrix
by this bijection.

Each module $M$ has a color $c\left(M\right)\in\left[n\right]=\left\{ 0,...,n-1\right\} $.
It is uniquely determined by $\det\left(A_{M}\right)=t^{c\left(M\right)+nz}$,
for some $z\in\N$. For every color $i$, the subgroup $\PSL_{n}\left(l_{0}\right)\le G$
acts transitively on vertices of color $i$.

Let $M_{0}$ be the standard module corresponding to the identity
matrix $I$, or the identity coset in $G/K$. Its neighbors are the
modules which corresponding to the set $N=\left\{ A_{M_{1}},...,A_{M_{t}}\right\} $
of matrices $A_{M_{i}}$ as above, which further satisfy:
\begin{enumerate}
\item $0\le m_{1},...,m_{n}\le1$.
\item If $m_{j}=1$ then $f_{i,j}=0$ for $i<j$.
\item $A\ne I$.
\end{enumerate}
Finally, for every module $M$, its neighbors correspond to matrices
of the form $\left\{ A_{M}A:A\in N\right\} $, up to dividing by a
power of $t$ which is the gcd of all the elements.

Given a lattice $\Gamma\le\PGL_{n}\left(l_{0}\right)$, we may look
at the quotient space $\Gamma\backslash B_{G}$. Assuming that $\Gamma$
does not intersect a big enough neighborhood of the identity (and,
in particular, is torsion-free), $\Gamma\backslash B_{G}$ is a simplicial complex. In the case $n=2$ this is a $\left(q+1\right)$-regular
graph. We again refer to \cite{lubotzky2014ramanujan} for a discussion
of such complexes. If we have a lattice $\Gamma'\le\SL_{n}\left(l_{0}\right)$
we may project it to $\Gamma\le\PGL_{n}\left(l_{0}\right)$, and
then, assuming that $\Gamma'$ is torsion free, $\left|\Gamma'\backslash\SL_{n}\left(l_{0}\right)/\SL_{n}\left(\O\right)\right|=\frac{1}{n}\left|\Gamma\backslash B_{G}\right|$,
since $\SL_{n}\left(l_{0}\right)$ preserves the color of the vertices
of $B\left(G\right)$.

For the lattices $\Gamma$ constructed from division algebras in this
work, for every compact subset $S\subset G$, it holds that as $w_{1}$
changes, eventually $\Gamma\left(w_{1}\right)\cap S\in\left\{ e\right\} $.
This implies that $\Gamma\left(w_{1}\right)\backslash B_{G}$ will
indeed eventually be a simplicial complex. A far deeper fact is that
over function fields, $\Gamma\left(w_{1}\right)\backslash B_{G}$
is a \emph{Ramanujan complex} -- see \cite{lubotzky2005ramanujan,first2016ramanujan,kamber2016lpcomplex}
for a general discussion of this concept, and specifically \cite[Section 7]{first2016ramanujan}
for a proof (here we implicitly use the assumption that $n$ is prime).

\subsection{\label{subsec:Vertex-Expansion}Vertex Expansion}

Consider $l_{0}=\F_{q}\left(\left(t\right)\right)$ and its subfield
$k_{0}=\F_{q}\left(\left(t^{2}\right)\right)$. Let $\O_{k_{0}}=\F_{q}\left[\left[t^{2}\right]\right]$,
$\O_{l_{0}}=\F_{q}\left[\left[t\right]\right]$ be the corresponding
rings of integers.

We let $G=\PGL_{n}\left(l_{0}\right)$, $H=\PGL_{n}\left(k_{0}\right)$,
and $K=\PGL_{n}\left(\O_{l_{0}}\right)$, $K_{H}=H\cap K=\PGL_{n}\left(\O_{k_{0}}\right)$
the maximal compact open subgroups. We have an action of $H$ on Bruhat-Tits
building $B_{G}$ of $G$, and since the stabilizer of the standard
module is $K_{H}$, we have a map (on vertices) $F\colon B_{H}\to B_{G}$.
See Figure~\ref{fig:ramified extension tree} for a special case
of this embedding for $n=2$ and $q=2$.

We consider the set $N$ defining the neighbors in $B_{G}$ and the
set $N_{H}$ defining the neighbors in $B_{H}$, as in Subsection~\ref{subsec:Bruhat-Tits Buildings}.
There is a bijection $T\colon N\to N_{H}$, with $A\in N$ corresponding
to $T\left(A\right)\in N_{H}$ where $t$ is replaced by $s=t^{2}$.
Let $A\in N$ with diagonal $\left(t^{m_{1}},...,t^{m_{n}}\right)$,
then we get $T\left(A\right)$ by simply replacing the diagonal with
$\left(t^{2m_{1}},...,t^{2m_{n}}\right)$. Moreover, if $m_{j}=1$
then the $j$-th column is 0 outside the diagonal, and this implies
that $T\left(A\right)=AD\left(A\right)$, where $D\left(A\right)=\operatorname{diag}\left(t^{m_{1}},...,t^{m_{n}}\right)$
is the diagonal matrix with the same diagonal as $A$. Notice that
$D\left(A\right)\in N$. Therefore, if $A_{M}\in B_{H}$, then $F\left(A_{M}T\left(A\right)\right)=F\left(A_{M}\right)AD\left(A\right)$.
This discussion may be concluded as follows:
\begin{lem}
\label{lem:no unique neighbors}Let $M\in F\left(B_{H}\right)\subset B_{G}$,
and let $M'=MA$, $A\in N$ be a neighbor of $M$ in $B_{G}$. Then
there is $M''=M'D\left(A\right)=MT\left(A\right)\in F\left(B_{H}\right)$,
another neighbor of $M'$ from $F\left(B_{H}\right)$.

Therefore, F$\left(B_{H}\right)\subset B_{G}$ has no unique neighbors
(i.e., neighbors that are connected to it by a single edge).
\end{lem}

We may now prove Theorem~\ref{thm:Complexes}:
\begin{proof}[Proof of Theorem~\ref{thm:Complexes}]
By Proposition~\ref{prop:good pairs}, Theorem~\ref{thm:Good Pair Theorem},
and the discussion in Subsection~\ref{subsec:Bruhat-Tits Buildings},
there is a lattice $\Gamma\le\PGL_{n}\left(l_{0}\right)$ of arbitrarily
large covolume, such that if we denote $Y'=\Gamma\cap H\backslash B_{H}$,
$X=\Gamma\backslash B_{G}$, then $\left|Y'\right|=O\left(\left|X\right|^{1/2}\right)$
and $X$ and $Y'$ are Ramanujan complexes.

There is also a natural map $F_{\Gamma}\colon Y'\to X$. It holds that $Y=F_{\Gamma}\left(Y'\right)$
is the image of $F\left(B_{H}\right)$ under the projection $B_{G}\to\Gamma\backslash B_{G}=X$.
By Lemma~\ref{lem:no unique neighbors} $F\left(B_{H}\right)$ has
no unique neighbors in $B_{G}$. Therefore, the set $Y$ -- the projection
of $F\left(B_{H}\right)$ to $X$ -- has no unique neighbors.
\end{proof}

We may also complete the non-explicit part proof of Theorem~\ref{thm:vertex expansion intro} (i.e., for even $q$). 
\begin{proof}[Proof of Theorem~\ref{thm:vertex expansion intro}]
For $n=2$ the complex $X$ of Theorem~\ref{thm:Complexes} is a graph, and there is a subset $Y \subset X$, $|Y|= O(\sqrt{|X|})$, such that each $y\in N\left(Y\right)$ is connected to at least two vertices of $Y$. Then we can show that every
$y\in N\left(Y\right)$ is connected to precisely $2$ vertices of
$N\left(Y\right)$, as in the proof of Lemma~\ref{lem:Symmetry lemma} (see also the proof of Theorem~\ref{thm:edge expansion} in Subsection~\ref{subsec:Edge-Expansion}). 
\end{proof}

\subsection{\label{subsec:Edge-Expansion}Edge Expansion}

Let us now take $n=2$, $l_{0}=\F_{q^{2}}\left(\left(t\right)\right)$
and $G=\PGL_{2}\left(l_{0}\right)$. Let $k_{0}=\F_{q}\left(\left(t\right)\right)$,
and $H=\PGL_{2}\left(k_{0}\right)$, which is a subgroup of $G$.

The Bruhat-Tits building $B_{G}$ of $G$, which is described in the
previous section, is a $\left(q^{2}+1\right)$-regular tree. Explicitly,
the set $N$ determining the neighbors contains $\left(\begin{array}{cc}
1 & 0\\
0 & t
\end{array}\right)$ and $\left(\begin{array}{cc}
t & a\\
0 & 1
\end{array}\right)$ for $a\in\F_{q^{2}}$.

The subgroup $H$ acts on $B_{G}$, and the stabilizer of the standard
module $M_{0}$ is $H\cap\PGL_{2}\left(\F_{q^{2}}\left[\left[t\right]\right]\right)=\PGL_{2}\left(\F_{q}\left[\left[t\right]\right]\right)$,
which is a maximal compact open subgroup of $H$. We therefore have
a map $F:B_{H}\to B_{G}$. Let $N_{H}$ be the set determining the
neighbors in $B_{H}$, which contains $\left(\begin{array}{cc}
1 & 0\\
0 & t
\end{array}\right)$ and $\left(\begin{array}{cc}
t & a\\
0 & 1
\end{array}\right)$ for $a\in\F_{q}$. Notice that $N_{H}\subset H$. Moreover, for $A\in N_{H}$
and $M\in B_{H}$ described by its matrix, $F\left(MA\right)=F\left(M\right)A$.
In other words, adjacent vertices in $B_{H}$ are sent to adjacent
vertices in $B_{G}$ (see Figure~\ref{fig:edge expansion} for a
special case).

The discussion above implies:
\begin{lem}
\label{lem:edge exapnsion embedding lemma}Every vertex in $F\left(B_{H}\right)\subset B_{G}$
is connected to $q+1$ other vertices of $F\left(B_{H}\right)$.
\end{lem}

We may now prove Theorem~\ref{thm:edge expansion} from the introduction:
\begin{proof}[Proof of Theorem~\ref{thm:edge expansion}]
By Proposition~\ref{prop:good pairs}, Theorem~\ref{thm:Good Pair Theorem},
and the discussion in Subsection~\ref{subsec:Bruhat-Tits Buildings},
there is a lattice $\Gamma\le\PGL_{2}\left(l_{0}\right)$ of arbitrarily large covolume, such that if we denote $Y'=\Gamma\cap H\backslash B_{H}$,
$X=\Gamma\backslash B_{G}$, then $\left|Y'\right|=O\left(\left|X\right|^{1/2}\right)$
and $X$ and $Y'$ are Ramanujan graphs.

There is also a natural map $F_{\Gamma}\colon Y'\to X$. By Lemma~\ref{lem:edge exapnsion embedding lemma},
every vertex of $Y=F_{\Gamma}\left(Y'\right)$ is connected to at
least $q+1$ other vertices of $Y$, as $Y=F_{\Gamma}\left(Y\right)$
is the image of $F\left(Y\right)$ by the projection map $B_{G}\to\Gamma\backslash B_{G}$.

We claim that $Y$ is $\left(q+1\right)$-regular. This also implies
that $Y$ is a quotient of $Y'$, and is therefore a Ramanujan graph.

If $Y$ is not $\left(q+1\right)$-regular, there is a vertex $y\in Y$
that is connected to more than $\left(q+1\right)$ other elements
of $Y$. We next use symmetry as in the proof of Lemma~\ref{lem:Symmetry lemma}.
Notice that the group $\Gamma\left(w_{1}\right)\backslash\Gamma$
acts on $X$, and its subgroup $\Gamma'\left(v_{1}\right)\backslash\Gamma'$
(using the natural embedding $\Gamma'\left(v_{1}\right)\backslash\Gamma'\to\Gamma\left(w_{1}\right)\backslash\Gamma$)
preserves $Y$. Therefore, there are $\Theta\left(\left|\Gamma'\left(v_{1}\right)\backslash\Gamma'\right|\right)=\Theta\left(\left|Y\right|\right)$
vertices in $Y$ that are connected to more than $q+1$ vertices.
Since the minimal degree of $Y$ is $q+1$, the average degree is
greater than $q+1+\delta$ for some explicit $\delta>0$. This is
impossible by Kahale's Theorem~\ref{thm:Kahale edge expansion intro}.
\end{proof}

\section{\label{sec:Moore Bound}Expansion Using Moore's Bound}

In this section, we reprove Kahale's lower bounds about vertex and edge expansion in Ramanujan graphs. While the bounds we get are a bit weaker than Kahale's, we believe that they are easier to understand.

First, let us set notations. An undirected graph is a finite set $X$
of vertices, a finite set $E$ of directed edges, two maps $s,t\colon E\to X$,
and an involution $\overline{\cdot}\colon E\to E$, satisfying the conditions
$s\left(\overline{e}\right)=t\left(e\right)$, $\overline{e}\ne e$.
We allow multiple edges and self loops, but no ``half edges''. For
$x\in X$ we let $d_{x}=\#\left\{ e\in E:s\left(e\right)=x\right\} $
be the degree of $x$. We assume that $X$ is connected.

A non-backtracking path of length $l$ in $X$ is a sequence $\left(e_{1},...,e_{l}\right)$
of edges, with $t\left(e_{i}\right)=s\left(e_{i+1}\right)$ and $e_{i+1}\ne\overline{e_{i}}$.
We denote by $M_{l}\left(X\right)$ the number of non-backtracking
paths of length $l$ in $X$.

Given a subset $S\subseteq X$, we have an induced graph on $S$,
containing all the edges $e\in E$ with $s\left(e\right),t\left(e\right)\in S$.
Therefore, $M_{l}\left(S\right)$ is well-defined. We also denote by
$M_{l}\left(S,X\right)$ the non-backtracking paths $\left(e_{1},...,e_{l}\right)$
in $X$ such that $s\left(e_{1}\right),t\left(e_{l}\right)\in S$.

Our proofs uses the results of \cite{alon2002moore}, whose main technical
result is:
\begin{thm}[\cite{alon2002moore}]
\label{thm:Moore bound}Let $X$ be an undirected graph with $m$
directed edges, and assume that $X$ has no vertices of degree $1$.
Let
\[
\tilde{d}-1=\left(\prod_{e\in E}\left(d_{s\left(e\right)}-1\right)\right)^{1/m}=\left(\prod_{x\in X}\left(d_{x}-1\right)^{d_{x}}\right)^{1/m},
\]
i.e., the geometric average over the edges of the degree of their
source vertex minus $1$. Then
\[
M_{l}\left(X\right)\ge m\left(\tilde{d}-1\right)^{l-1}.
\]
\end{thm}

While the number $\tilde{d}$ is somewhat complicated, it holds:
\begin{lem}
\label{lem:d-comparison}Assume that $X$ is a graph without vertices
of degree $1$. Then:
\begin{enumerate}
\item (\cite{alon2002moore}) $\tilde{d}\ge\overline{d}$, where $\overline{d}$
is the average degree of $X$. \\
Moreover, for every $C>0$ and $\epsilon>0$ there exists $\delta>0$
such that if $\tilde{d}\le\overline{d}+\delta$ and $\overline{d}\le C$,
then there is an integer $d\ge2$ satisfying $\left|\tilde{d}-d\right|\le\epsilon$,
and all but $\epsilon n$ of the vertices of $X$ are of degree $d$.
\item If $X$ is bipartite, $\tilde{d}-1\ge\sqrt{\left(\overline{d}_{L}-1\right)\left(\overline{d}_{R}-1\right)}$,
where $\overline{d}_{L}$ (resp. $\overline{d}_{R}$) is the average
degree of the left side (resp. the right side) of $X$.
\end{enumerate}
\end{lem}

\begin{rem}
A similar ``moreover'' argument is true for the bipartite case,
but we will not need it.
\end{rem}

\begin{proof}
For the first claim, it holds that $m=\overline{d}n$. Therefore $\tilde{d}-1=\left(\prod_{x\in X}\left(d_{x}-1\right)^{d_{x}}\right)^{1/\overline{d}n}$.
So
\[
\log\left(\tilde{d}-1\right)=\frac{1}{\overline{d}}\frac{1}{n}\sum_{x\in X}d_{x}\log\left(d_{x}-1\right)\ge\frac{1}{\overline{d}}\overline{d}\log\left(\overline{d}-1\right)=\log\left(\overline{d}-1\right),
\]
where the inequality follows from the convexity of $d\log\left(d-1\right)$
for $d\ge2$. The ``moreover'' part follows from the strict convexity
of $d\log\left(d-1\right)$ for $d\ge2$.

For the second claim, let $n_{L},n_{R}$ be the number of vertices
in $X_{L},X_{R}$ -- the right and left sides of $X$. Then
\[
n_{L}\overline{d}_{L}=n_{R}\overline{d}_{R}=m/2.
\]
Therefore
\begin{align*}
\tilde{d}-1 & =\left(\prod_{x\in X}\left(d_{x}-1\right)^{d_{x}}\right)^{1/\overline{d}n}\\
 & =\left(\prod_{x\in X_{L}}\left(\left(d_{x}-1\right)^{d_{x}}\right)^{1/\overline{d}_{L}n_{L}}\right)^{1/2}\left(\prod_{x\in X_{R}}\left(\left(d_{x}-1\right)^{d_{x}}\right)^{1/\overline{d}_{R}n_{R}}\right)^{1/2}\\
 & \le\sqrt{\left(\overline{d}_{L}-1\right)\left(\overline{d}_{R}-1\right)},
\end{align*}
where the inequality is as in the proof of the first claim.
\end{proof}
We combine the lower bound on $M_{l}\left(S\right)$ coming from Moore's
bound, with a standard upper bound from spectral graph theory, which we now describe.

We let $L^{2}\left(X\right)$ be the set of functions $f\colon X\to\C$,
with the usual inner product. The adjacency operator $A\colon L^{2}\left(X\right)\to L^{2}\left(X\right)$
is defined as
\[
\left(Af\right)\left(x\right)=\sum_{e\in E,t\left(e\right)=x}f\left(s\left(e\right)\right).
\]

The operator $A$ is self-adjoint and therefore diagonalizable, with
real eigenvalues and an orthogonal basis of eigenvectors. If $X$
is $d$-regular, the constant function is an eigenvector of $A$,
with eigenvalue $d$. If $X$ is bipartite, $A$ has the eigenvalue
$-d$, corresponding to the eigenvector that is equal to a constant
$C$ on one part, and is equal to $-C$ on the other part. If $X$
is connected, as we assume, those are the only eigenvectors with eigenvalues
of absolute value $d$. We say that $X$ is Ramanujan if every eigenvalue
$\lambda$ of $A$ satisfies either $\left|\lambda\right|=d$ or $\left|\lambda\right|\le2\sqrt{d-1}$.
This bound is optimal for large graphs by the Alon-Boppana Theorem (\cite{nilli1991second}).
\begin{lem}
\label{lem:Walks bound}Let $X$ be a $d$-regular Ramanujan graph
with $n$ vertices, and let $S\subset X$ be a subset.

For $l$ such that $\left|S\right|\left(d-1\right)^{l/2}\le n$, the
number $M_{l}\left(S,X\right)$ of non-backtracking paths in $X$
that start and end in $S$ satisfies that $M_{l}\left(S,X\right)\le\left|S\right|\left(l+3\right)\left(d-1\right)^{l/2}$.
\end{lem}

\begin{proof}
We define a length $l$ non-backtracking version of the adjacency
operator, $A_{l}:L^{2}\left(X\right)\to L^{2}\left(X\right)$,
\[
\left(A_{l}f\right)\left(x\right)=\sum_{\left(e_{1},...,e_{l}\right),t\left(e_{l}\right)=x}f\left(s\left(e_{1}\right)\right),
\]
where the sum is over the non-backtracking paths in $X$. Notice that
it holds that $M_{l}\left(S,X\right)=\left\langle A_{l}\mathds{1}_{S},\mathds{1}_{S}\right\rangle $,
where $\one_{S}$ is the characteristic function of $S$.

Since the graph is $d$-regular, there is a simple relation between
the $A_{l}$-s, given by
\begin{align*}
A & =A_{1}\\
A^{2} & =dI+A_{2}\\
AA_{l} & =\left(d-1\right)A_{l-1}+A_{l+1}\,\,\,\,\,\,\,\,l>1.
\end{align*}

The relations imply that $A_{l}$ is a polynomial in $A$, given
explicitly for $l\ge2$ by
\[
A_{l}=\left(d-1\right)^{l/2}\left(\left(1-\left(d-1\right)^{-1}\right)U_{l}\left(\frac{A}{2\sqrt{d-1}}\right)+2\left(d-1\right)^{-1}T_{l}\left(\frac{A}{2\sqrt{d-1}}\right)\right),
\]
where $T_{l}$ and $U_{l}$ are the Chebyshev polynomials of the first
and second kind, given by $U_{l}\left(\cos\theta\right)=\frac{\sin\left(\left(n+1\right)\theta\right)}{\sin\theta}$,
$T_{l}\left(\cos\theta\right)=\cos\left(n\theta\right)$. In particular,
if $f\in L^{2}\left(X\right)$ is an eigenvector of $A$, $f$ is
also an eigenvector of $A_{l}$. If the eigenvalue of $A$ is bounded
in absolute value by $2\sqrt{d-1}$, the corresponding eigenvalue
of $A_{l}$ is bounded in absolute value by $\left(l+1\right)\left(d-1\right)^{l/2}$.

Returning to the relation $M_{l}\left(S,X\right)=\left\langle A_{l}\mathds{1}_{S},\mathds{1}_{S}\right\rangle $,
first assume that $X$ is non-bipartite. We write $\mathds{1}_{S}=\frac{\left|S\right|}{n}\mathds{1}_{X}+r$,
with $r\perp\mathds{1}_{X}$ and $\n r_{2}^{2}\le\n{\one_{S}}_{2}^{2}=\left|S\right|$.
Notice that $A_{l}\one_{X}=d\left(d-1\right)^{l}\one_{X}$. By the
Ramanujan assumption and the fact that $A_{l}$ is self-adjoint,
\[
\left|\left\langle A_{l}r,r\right\rangle \right|\le\left(l+1\right)\left(d-1\right)^{l/2}\n r_{2}^{2}.
\]

Then
\begin{align*}
M_{l}\left(S,X\right) & =\left\langle A_{l}\mathds{1}_{S},\mathds{1}_{S}\right\rangle \\
 & =\frac{\left|S\right|^{2}}{n^{2}}\left\langle A_{l}\mathds{1}_{X},\mathds{1}_{X}\right\rangle +\left\langle A_{l}r,r\right\rangle \\
 & \le\frac{\left|S\right|^{2}}{n^{2}}nd\left(d-1\right)^{l}+\left(l+1\right)\left(d-1\right)^{l/2}\n r_{2}^{2}\\
 & \le\left|S\right|\left(\frac{\left|S\right|d\left(d-1\right)^{l-1}}{n}+\left(l+1\right)\left(d-1\right)^{l/2}\right)\\
 & \le\left|S\right|\left(\frac{d}{d-1}+l+1\right)\left(d-1\right)^{l/2}.
\end{align*}

We conclude by noting that $\frac{d}{d-1}\le2$. The case of bipartite
graphs is similar.
\end{proof}
We can compare $M_{l}\left(S,X\right)$ with $M_{l}\left(S\right)$,
as it is obvious that
\[
M_{l}\left(S,X\right)\ge M_{l}\left(S\right).
\]

The case of edge expansion essentially follows directly:
\begin{thm}
\label{thm:Kahale edge expansion}Let $X$ be a $d$-regular Ramanujan
graph and let $S\subset X$ be a subset with $\left|S\right|\left(d-1\right)^{l/2}\le\left|X\right|$.
Then the average degree of the graph $S$ induced from $X$ is bounded
by $\left(1+O\left(\frac{\ln\left(l+3\right)}{l}\right)\right)\sqrt{d-1}+1$.

Moreover, assuming $\left|S\right|=o\left(\left|X\right|\right)$,
for $\delta>0$ small enough, there is $\epsilon>0$, such that if
the average degree of $\left|S\right|$ is larger than $\sqrt{d-1}+1+\epsilon$,
then $\sqrt{d-1}+1$ is an integer and at most $\delta\left|S\right|$
of the vertices of $\left|S\right|$ have a degree different from $\sqrt{d-1}+1$.
\end{thm}

\begin{proof}
We may assume that the average degree of $S$ is at least $2$. We
then may remove from $S$ vertices of degree $1$ without lowering
the average degree, until all the degrees are at least $2$. Notice
that if we remove more than $o\left(\left|S\right|\right)$ of the
vertices, then the average degree grows by a constant.

By Theorem~\ref{thm:Moore bound},
\[
M_{l}\left(S\right)\ge
\left|S\right|\overline{d}\left(\tilde{d}-1\right)^{l-1}\ge
\left|S\right|\left(\tilde{d}-1\right)^{l}(d-1)^{-1},
\]
where $\tilde{d}$ is as in the theorem.

On the other hand, by Lemma~\ref{lem:Walks bound},
\[
M_{l}\left(S,X\right)\le\left|S\right|\left(l+3\right)\left(d-1\right)^{l/2}.
\]

Therefore,
\[
\tilde{d}-1\le\left(l+3\right)^{1/l}\left(d-1\right)^{1/2}=\left(1+O\left(\ln\left(l+3\right)/l\right)\right)\left(d-1\right)^{1/2}.
\]

So by Lemma~\ref{lem:d-comparison}, $\overline{d}\le\tilde{d}\le\left(1+O\left(\ln\left(l+3\right)/l\right)\right)\sqrt{d-1}+1$.

The ``moreover'' part follows from the ``moreover'' part of Lemma~\ref{lem:d-comparison}.
\end{proof}
The vertex expansion result is similar, but one should be a bit more
careful when handling vertices of degree $1$.
\begin{thm}
\label{thm:Kahale vertex expansion}Let $X$ be a $d$-regular Ramanujan
graph and let $S\subset X$ be a subset with $\left|S\right|\left(d-1\right)^{l}\le\left|X\right|$.
Let $N\left(S\right)$ be the neighbors of $\left|S\right|$. Then
\[
\left|N\left(S\right)\right|\ge\frac{d}{2}\left|S\right|\left(1-O\left(\frac{\ln\left(l+3\right)}{l}\right)\right).
\]
Moreover, assuming that $\left|S\right|=o\left(\left|X\right|\right)$,
for every $\epsilon>0$ there is $\delta>0$, such that for $\left|X\right|$
large enough, if $\left|N\left(S\right)\right|\le\frac{d}{2}\left|S\right|\left(1+\delta\right)$
then all but at most $\epsilon\left|N\left(S\right)\right|$ of the
vertices of $\left|N\left(S\right)\right|$ are connected to exactly
$2$ vertices of $S$.
\end{thm}

\begin{proof}
We assume that $X$ is bipartite and $S\subset X$ is contained in
one of the sides. See the proof of Theorem 2 in \cite{kahale1995eigenvalues}
for this simple reduction.

Decompose $N\left(S\right)=M\cup M'$, where $M$ are vertices that
are connected to two or more vertices in $S$ and $M'$ are vertices
that are connected to exactly one vertex in $S$. We may assume that
every vertex in $S$ is connected to at least $2$ vertices in $M$.
Otherwise, assuming the ratio $\left|N\left(S\right)\right|/\left|S\right|$
is smaller than $d-1$, when we remove a vertex that is connected
to one or zero vertices in $M$, we decrease $S$ by $1$ and decrease $N(S)$ by at least $d-1$, so we decrease the ratio $\left|N\left(S\right)\right|/\left|S\right|$.

Consider the bipartite graph $Y$ on $\left(S,M\right)$, where $S$
is on the left side and $M$ is on the right side. Let $m=\left|M\right|$,
$m'=\left|M'\right|$, $s=\left|S\right|$. Let $e$ be the number
of directed edges from $S$ to $M$ (notice that it is half of the
edges in $Y$, which contain edges from $M$ to $S$ as well).

It holds that
\begin{align*}
m' & =ds-e\\
\left|N\left(S\right)\right| & =m+m'\\
 & =m+ds-e.
\end{align*}

The average left degree of $Y$ is $\overline{d}_{L}=\frac{e}{s}$
and the average right degree is $\overline{d}_{R}=\frac{e}{m}$. Write
$d'=\sqrt{\left(\overline{d}_{L}-1\right)\left(\overline{d}_{R}-1\right)}+1$.

By Theorem~\ref{thm:Moore bound} and Lemma~\ref{lem:d-comparison}, \[
M_{l}(Y)\ge |Y|\left(d'-1\right)^{l-1}.
\]

However, by Lemma~\ref{lem:Walks bound},
\[
M_{l}(Y)\le |Y|(l+3)\left(d-1\right)^{l/2}..
\]

Denote $\epsilon=\ln\left(l+3\right)/l$. Then we get, as before,
\[
d'-1\le\sqrt{d-1}\left(1+O\left(\epsilon\right)\right)
\]

Therefore
\[
\sqrt{\frac{e}{s}-1}\sqrt{\frac{e}{m}-1}\le\sqrt{d-1}\left(1+O\left(\epsilon\right)\right)
\]

Simplifying,

\begin{align*}
\frac{e}{m} & \le1+\frac{\left(1+O\left(\epsilon\right)\right)\left(d-1\right)s}{e-s}=\frac{e+\left(d-2\right)s}{e-s}\left(1+O\left(\epsilon\right)\right)\\
m & \ge\frac{e\left(e-s\right)}{e+\left(d-2\right)s}\left(1-O\left(\epsilon\right)\right)\\
m-e & \ge-\frac{e\left(d-1\right)s}{e+\left(d-2\right)s}\left(1+O\left(\epsilon\right)\right).
\end{align*}

Since $e\le ds$, we get
\[
m-e\ge-\frac{d}{2}s\left(1+O\left(\epsilon\right)\right).
\]

and since $\left|N\left(S\right)\right|=m-e+ds$, we deduce
\[
\left|N\left(S\right)\right|\ge\frac{d}{2}s\left(1-O\left(\epsilon\right)\right).
\]

The proof also says that if $e\le\alpha ds$ for some fixed $\alpha<1$,
then
\[
|N\left(S\right)|\ge\beta\frac{d}{2}s\left(1-O(\epsilon)\right)
\]
for some $\beta>1$ depending on $\alpha$. Therefore, if we assume
that $|S|=o(|N|)$ and $\left|N\left(S\right)\right|\le\frac{d}{2}s\left(1+o\left(1\right)\right)$, then $e\ge ds\left(1-o\left(1\right)\right)$. Therefore, all but $o\left(s\right)$ of the vertices of $N\left(S\right)$ are connected to at least $2$ vertices of $S$, and by the bound on the size of $N\left(S\right)$, all but $o\left(s\right)$ of the vertices of $N\left(S\right)$ are
connected to exactly $2$ vertices of $S$.
\end{proof}
\begin{rem}
The proofs of Kahale give slightly better bounds for both edge and
vertex expansions, where $O\left(\ln\left(l+3\right)/l\right)$ is
replaced by $O\left(1/l\right)$.
\end{rem}

\bibliographystyle{acm}
\bibliography{./database}

\end{document}